\documentclass[12pt]{amsart}
\usepackage[utf8]{inputenc}
\usepackage{amssymb}
\usepackage{hyperref}
\usepackage[final]{showkeys} 

\input xy
\xyoption{all}

\theoremstyle{definition}
\newtheorem{mydef}{Definition}[section]
\newtheorem{lem}[mydef]{Lemma}
\newtheorem{thm}[mydef]{Theorem}
\newtheorem{conjecture}[mydef]{Conjecture}
\newtheorem{cor}[mydef]{Corollary}

\newtheorem{hypothesis}[mydef]{Hypothesis}

\newtheorem{defin}[mydef]{Definition}

\newtheorem{remark}[mydef]{Remark}
\newtheorem{notation}[mydef]{Notation}
\newtheorem{fact}[mydef]{Fact}

\newcommand{\fct}[2]{{}^{#1}#2}

\newcommand{\ba}{\bar{a}}
\newcommand{\bb}{\bar{b}}
\newcommand{\bc}{\bar{c}}

\newcommand{\bx}{\bar{x}}
\newcommand{\by}{\bar{y}}
\newcommand{\bz}{\bar{z}}

\newcommand{\sea}{\mathfrak{C}}

\newcommand{\ran}[1]{\text{ran}(#1)}

\newcommand{\cf}[1]{\text{cf} (#1)}
\newcommand{\seq}[1]{\langle #1 \rangle}
\newcommand{\rest}{\upharpoonright}

\newcommand{\leap}[1]{\le_{#1}}

\newcommand{\geap}[1]{\ge_{#1}}

\newcommand{\lea}{\leap{\K}}

\newcommand{\gea}{\geap{\K}}

\def\lee{\preceq}

\newcommand{\K}{\mathbf{K}}

\newbox\noforkbox \newdimen\forklinewidth
\forklinewidth=0.3pt \setbox0\hbox{$\textstyle\smile$}
\setbox1\hbox to \wd0{\hfil\vrule width \forklinewidth depth-2pt
 height 10pt \hfil}
\wd1=0 cm \setbox\noforkbox\hbox{\lower 2pt\box1\lower
2pt\box0\relax}
\def\unionstick{\mathop{\copy\noforkbox}\limits}
\newcommand{\nf}{\unionstick}
\newcommand{\nfs}[4]{#2 \nf_{#1}^{#4} #3}

\def\1nf{\unionstick^{(1)}}

\def\2nf{\unionstick^{(2)}}
\def\3nf{\unionstick^{(3)}}

\def\nfm{\overline{\nf}}
\newcommand{\nfcl}[4]{#2 \overset{#4}{\underset{#1}{\overline{\nf}}} #3}

\newcommand{\tp}{\text{tp}}
\newcommand{\gtp}{\text{gtp}}

\newcommand{\gS}{\text{gS}}

\newcommand{\hanf}[1]{h (#1)}
\newcommand{\ehanf}[1]{\beth_{\left(2^{#1}\right)^+}}

\newcommand{\Axfr}{\text{AxFri}_1}

\newcommand{\Aut}{\operatorname{Aut}}
\newcommand{\EM}{\operatorname{EM}}
\newcommand{\Ll}{\mathbb{L}}
\newcommand{\otp}{\operatorname{otp}}
\newcommand{\pred}{\operatorname{pred}}
\newcommand{\topp}{\operatorname{top}}

\newcommand{\tlt}{\triangleleft}
\newcommand{\tleq}{\trianglelefteq}

\newcommand{\cl}{\operatorname{cl}}

\newcommand{\LS}{\text{LS}}

\newcommand{\BI}{\mathbf{I}}

\newcommand{\Av}{\text{Av}}

\title[Categoricity in universal classes. Part II]{Shelah's eventual categoricity conjecture in universal classes. Part II}
\date{\today}

\author{Sebastien Vasey}
\email{sebv@cmu.edu}
\urladdr{http://math.cmu.edu/\textasciitilde svasey/}
\address{Department of Mathematical Sciences, Carnegie Mellon University, Pittsburgh, Pennsylvania, USA}
\thanks{This material is based upon work done while the author was supported by the Swiss National Science Foundation under Grant No.\ 155136.}

\begin{document}

\begin{abstract}
  We prove that a universal class categorical in a high-enough cardinal is categorical on a tail of cardinals. As opposed to other results in the literature, we work in ZFC, do not require the categoricity cardinal to be a successor, do not assume amalgamation, and do not use large cardinals. Moreover we give an explicit bound on the ``high-enough'' threshold:
  
  \begin{thm}\label{abstract-thm}
    Let $\psi$ be a universal $\Ll_{\omega_1, \omega}$ sentence (in a countable vocabulary). If $\psi$ is categorical in \emph{some} $\lambda \ge \beth_{\beth_{\omega_1}}$, then $\psi$ is categorical in \emph{all} $\lambda' \ge \beth_{\beth_{\omega_1}}$.
  \end{thm}

  As a byproduct of the proof, we show that a conjecture of Grossberg holds in universal classes:

  \begin{cor}\label{abstract-thm-2}
    Let $\psi$ be a universal $\Ll_{\omega_1, \omega}$ sentence (in a countable vocabulary) that is categorical in some $\lambda \ge \beth_{\beth_{\omega_1}}$, then the class of models of $\psi$ has the amalgamation property for models of size at least $\beth_{\beth_{\omega_1}}$.
  \end{cor}
  
  We also establish generalizations of these two results to uncountable languages. As part of the argument, we develop machinery to transfer model-theoretic properties between two different classes satisfying a compatibility condition (agreeing on any sufficiently large cardinals in which either is categorical). This is used as a bridge between Shelah's milestone study of universal classes (which we use extensively) and a categoricity transfer theorem of the author for abstract elementary classes that have amalgamation, are tame, and have primes over sets of the form $M \cup \{a\}$.
\end{abstract}

\maketitle

\tableofcontents

\section{Introduction}

In 1965, Morley \cite{morley-cip} started what is now called stability theory by proving:

\begin{fact}
  If a countable first-order theory is categorical\footnote{We say that a class of structures is \emph{categorical} in a cardinal $\lambda$ if it has a unique (up to isomorphism) model of size $\lambda$. We say that a theory or sentence (in some logic) is categorical in $\lambda$ if its class of models is.} in \emph{some} uncountable cardinal, then it is categorical in \emph{all} uncountable cardinals.
\end{fact}

In 1976, Shelah proposed \cite[Open Problem D.(3a)]{shelahfobook} the following far-reaching generalization:

\begin{conjecture}[Shelah's categoricity conjecture for $\Ll_{\omega_1, \omega}$]\label{categ-lomega1}
  Let $\psi$ be an $\Ll_{\omega_1, \omega}$ sentence. If $\psi$ is categorical in \emph{some} cardinal $\lambda \ge \beth_{\omega_1}$, then $\psi$ is categorical in \emph{all} cardinals $\lambda' \ge \beth_{\omega_1}$.
\end{conjecture}

This is now recognized as the central test question in nonelementary model theory. In 1977, Shelah introduced abstract elementary classes (AECs) \cite{sh88}, an abstract framework encompassing classes of models of an $\Ll_{\lambda^+, \omega} (Q)$ theory and several other examples of interest. Shelah has stated in \cite[N.4.2]{shelahaecbook2} the following version of the conjecture:

\begin{conjecture}[Shelah's eventual categoricity conjecture for AECs]\label{categ-aec}
  If an AEC is categorical in a high-enough cardinal, then it is categorical on a tail of cardinals.
\end{conjecture}

While many pages of approximations exist (see the references given after the statement of the main theorem below) both conjectures are still open. 

In this paper, we prove an approximation of Conjecture \ref{categ-lomega1} when $\psi$ is a \emph{universal} (see Definition \ref{univ-def}) sentence ($\beth_{\omega_1}$ is replaced by $\beth_{\beth_{\omega_1}}$, see more below). More generally, we confirm Conjecture \ref{categ-aec} for universal classes: classes of models of a universal $\Ll_{\infty, \omega}$ theory, or equivalently classes of models $K$ in a fixed vocabulary $\tau (K)$ closed under isomorphisms, substructure, and unions of $\subseteq$-increasing chains.

\textbf{Main Theorem \ref{main-thm}.}
  Let $K$ be a universal class. If $K$ is categorical in \emph{some} $\lambda \ge \beth_{\beth_{\left(2^{|\tau (K)| + \aleph_0}\right)^+}}$, then $K$ is categorical in \emph{all} $\lambda' \ge \beth_{\beth_{\left(2^{|\tau (K)| + \aleph_0}\right)^+}}$.

  Let us compare the main theorem to earlier approximations to Shelah's eventual categoricity conjecture\footnote{We do not present a complete history or an exhaustive list of recent results here. See the introduction of \cite{ap-universal-v9} for the former and \cite{bv-survey-v4-toappear} for the latter.}: In a series of papers \cite{tamenessone, tamenesstwo, tamenessthree}, Grossberg and VanDieren isolated \emph{tameness}, a locality properties of AECs, and (using earlier work of Shelah \cite{sh394}) proved Shelah's eventual categoricity conjecture in tame AECs with amalgamation assuming that the starting categoricity cardinal is a successor. Boney \cite{tamelc-jsl} later showed (building on work of Makkai-Shelah \cite{makkaishelah}) that tameness (as well as amalgamation, if in addition categoricity in a high-enough cardinal is assumed) follows from a large cardinal axiom (a proper class of strongly compact cardinals exists). Therefore the eventual categoricity conjecture follows from the following two extra assumptions: the categoricity cardinal is a successor, and a large cardinal axiom holds. In \cite[IV.7.12]{shelahaecbook}, Shelah removes the successor hypothesis assuming amalgamation\footnote{By \cite[Theorem 7.6]{tamelc-jsl}, this can also be replaced by a large cardinal axiom.} and the generalized continuum hypothesis (GCH)\footnote{It is enough to assume the existence of a suitable family of cardinals $\theta$ such that $2^{\theta} < 2^{\theta^+}$.}. Shelah's proof is clarified in \cite[Section 11]{downward-categ-tame-v6-toappear}, but it relies on a claim which Shelah has yet to publish a proof of.

In any case, all known categoricity transfers (which do not make model-theoretic assumptions on the AEC) rely on the existence of large cardinals together with either GCH or the assumption that the categoricity cardinal is a successor.

In the prequel to this paper \cite{ap-universal-v9} we showed that some of these limitations could be overcome in the case of universal classes\footnote{In earlier versions of \cite{ap-universal-v9} we claimed to prove the main theorem here but a mistake was later discovered.}:

\begin{fact}[Corollary 5.27 in \cite{ap-universal-v9}]\label{univ-fact}
  Let $K$ be a universal class.
  \begin{enumerate}
    \item\label{univ-fact-1} If $K$ is categorical in cardinals of arbitrarily high cofinality, then $K$ is categorical on a tail of cardinals.
    \item\label{univ-fact-2} If $\kappa > |\tau (K)| + \aleph_0$ is a measurable cardinal and $K$ is categorical in \emph{some} $\lambda \ge \beth_{\beth_{\beth_\kappa}}$ then $K$ is categorical in \emph{all} $\lambda' \ge \beth_{\beth_{\beth_\kappa}}$.
  \end{enumerate}
\end{fact}

Still, requirements on the categoricity cardinal in the first case and the existence of large cardinals in the second case could not be completely eliminated. These hypotheses were made to prove the amalgamation property, which is known to be the only obstacle:

\begin{fact}[Corollary 10.11 in \cite{downward-categ-tame-v6-toappear}]\label{categ-univ-ap}
  Let $K$ be a universal class with amalgamation. If $K$ is categorical in \emph{some} $\lambda \ge \beth_{\left(2^{|\tau (K)| + \aleph_0}\right)^+}$, then $K$ is categorical in \emph{all} $\lambda' \ge \beth_{\left(2^{|\tau (K)| + \aleph_0}\right)^+}$.
\end{fact}

Note that (see \cite{ap-universal-v9, categ-primes-v4}) all the facts stated above hold in a much wider context than universal classes: tame AECs with primes. However for the specific case of universal classes there is a well-developed structure theory \cite{sh300-orig}. This paper uses it to remove the assumption of amalgamation from Fact \ref{categ-univ-ap} and prove the main theorem. Further, a conjecture of Grossberg \cite[Conjecture 2.3]{grossberg2002} says that any AEC categorical in a high-enough cardinal should have amalgamation on a tail. A byproduct of this paper is that Grossberg's conjecture holds in universal classes (see the proof of Theorem \ref{main-thm}). Note that the behavior of amalgamation in universal classes is nontrivial: Kolesnikov and Lambie-Hanson have shown \cite{coloring-classes-jsl} that for each $\alpha < \omega_1$, there is a universal class in a countable vocabulary that has amalgamation up to $\beth_\alpha$ but fails amalgamation everywhere above $\beth_{\omega_1}$ (the example is not categorical in any uncountable cardinal).

One might think that Grossberg's conjecture should be established \emph{before} transferring categoricity (in order to be able to assume amalgamation in the transfer), but our proof of Theorem \ref{main-thm} is more subtle. First we use Shelah's structure theory of universal classes to show that there exists an ordering $\le$ (potentially different from substructure) such that $(K, \le)$ has amalgamation and other structural properties. We then work inside $(K, \le)$ to transfer categoricity (proving Theorem \ref{main-thm} since its statement does not depend on the ordering of the class). It is only after that we are able to conclude that $\le$ is actually substructure (on a tail of cardinals), and hence that Grossberg's conjecture holds in universal classes.

The main difficulty in the argument just outlined is that it is unclear that $(K, \le)$ is an AEC (it may fail the smoothness axiom). The hard part of this paper is proving that it actually \emph{is} an AEC. This is done by working inside a framework for forking-like independence in $(K, \le)$ that Shelah calls $\Axfr$ and proving new results for that framework, including Theorem \ref{tree-constr-thm} telling us how to copy a chain witnessing the failure of smoothness into an independent tree of models.

It should be noted that these new results (in Section \ref{enum-trees-sec}) are really the only new pieces needed to prove the main theorem. The rest of the paper is about combining the structure theory of universal classes developed by Shelah \cite[Chapter V]{shelahaecbook2} with known categoricity transfers \cite{ap-universal-v9, categ-primes-v4, downward-categ-tame-v6-toappear}. Another contribution of this paper is Section \ref{compat-sec} which considers two weak AECs $\K^1, \K^2$ satisfying a compatibility condition (the isomorphism types of models in a categoricity cardinal is the same). The motivation here is the aforementioned change from $\K^1 = (K, \subseteq)$ to $\K^2 = (K, \le)$:  In general, we may want to study an AEC $\K^1$ by changing its ordering, giving a new class $\K^2$ which has certain properties $P$ of $\K^1$ together with some new properties $P'$ that $\K^1$ may not have. We may know a theorem telling us that a single class that has both $P$ and $P'$ is well-behaved. Section \ref{compat-sec} gives tools to generalize the original theorem to the case when we do \emph{not} have a single class (i.e.\ $\K^1 = \K^2$) but instead have potentially different classes $\K^1$ and $\K^2$.

Note in passing that this paper does not make \cite{ap-universal-v9} obsolete: the results there hold for a wider context than universal classes, whereas we do not know how to generalize the proof of the main theorem here. Furthermore, we rely heavily here on \cite{ap-universal-v9}.

A natural question is why, the threshold in Theorem \ref{abstract-thm} is $\beth_{\beth_{\omega_1}}$ and not $\beth_{\omega_1}$ as in Conjecture \ref{categ-lomega1}. The $\beth_{\beth_{\omega_1}}$ comes from the fact that, in the argument outlined in the second paragraph after Fact \ref{categ-univ-ap}, the class $(K, \le)$ has Löwenheim-Skolem-Tarski number $\chi$, for some $\chi < \beth_{\omega_1}$. After proving that it is an AEC, we apply known categoricity transfers to this class, hence the final threshold for categoricity is of order $\beth_{(2^{\chi})^+} \le \beth_{\beth_{\omega_1}}$ (a similar phenomenon occurs in \cite{sh394}, where Shelah proves that the class $\K$ is $\chi$-weakly tame for some $\chi < \beth_{(2^{\LS (\K)})^+}$ and then obtains a threshold of $\beth_{(2^{\chi})^+}$). We do not know whether the threshold in Theorem \ref{abstract-thm} can be lowered to $\beth_{\omega_1}$.

Let us discuss the background required to read this paper. It is assumed that the reader has a solid knowledge of AECs (including at minimum the material in \cite{baldwinbook09}). Still, except for the basic concepts, we have tried to explicitly state all the definitions and facts. Only little understanding of \cite{ap-universal-v9, categ-primes-v4, downward-categ-tame-v6-toappear} is required: they are used only as black boxes. While some results in \cite{ap-universal-v9} rely on deep results of Shelah from the first sections of Chapter IV of \cite{shelahaecbook}, we do not use them\footnote{The one exception is \cite[IV.1.12.(2)]{shelahaecbook} (see Fact \ref{elem-facts}), but the proof is short and elementary.}.  At one point (Lemma \ref{stable-prop}) we rely on Shelah's construction of a certain linear order \cite[IV.5]{shelahaecbook}. This can also be taken as a black box. Last but not least, we rely on part of Shelah's original study of universal classes \cite{sh300-orig} (we quote from the updated version in Chapter V of \cite{shelahaecbook2}). All the results that we use from there have full proofs. We do not rely on any of Shelah's nonstructure results. 

\subsection{Acknowledgments}

This paper was written while working on a Ph.D.\ thesis under the direction of Rami Grossberg at Carnegie Mellon University and I would like to thank Professor Grossberg for his guidance and assistance in my research in general and in this work specifically.

I thank John Baldwin for inviting me to visit UIC in Fall 2015 to present a preliminary version of \cite{ap-universal-v9}. The present paper is an answer to several questions he asked me. This paper was also presented at seminars in Harvard and Rutgers University. I thank the organizers of these seminars for showing interest in my work and inviting me to talk. I thank the participants of these seminars for helpful feedback that helped me refine the presentation and motivation for this paper. I thank the referee for a detailed report that helped me improve the presentation of this paper.

Finally, this paper would not exist without the constant support and encouragements of Samaneh. I would like to dedicate this work to her.

\section{Preliminaries}

We state definitions and facts that will be used later. All throughout this paper, we use the letters $M, N$ for models and write $|M|$ for the universe of a model $M$ and $\|M\|$ for the cardinality of its universe. We may abuse notation and write e.g.\ $a \in M$ when we really mean $a \in |M|$.

Recall the definition of a universal class (for examples, see e.g.\ \cite[Example 2.2]{ap-universal-v9}).

\begin{defin}[\cite{tarski-th-models-i, sh300-orig}]
  A class of structure $K$ is \emph{universal} if:
  
  \begin{enumerate}
    \item It is a class of $\tau$-structures for a fixed vocabulary $\tau = \tau (K)$, closed under isomorphisms.
    \item If $\seq{M_i : i < \delta}$ is $\subseteq$-increasing in $K$, then $\bigcup_{i < \delta} M_i \in K$.
    \item If $M \in K$ and $M_0 \subseteq M$, then $M_0 \in K$.
  \end{enumerate}
\end{defin}
\begin{remark}
  Notice the following fundamental property of a universal class $K$. Given a subset $A$ of $N \in K$, $\cl^N (A)$, the closure of $A$ under the functions of $N$ (or equivalently $\bigcap\{N_0 \in K \mid A \subseteq |N|, N_0 \subseteq N\}$) is in $K$.
\end{remark}

It is known that universal classes can be characterized syntactically. We will use the following definition.

\begin{defin}\label{univ-def}
  A sentence $\psi$ of $\Ll_{\infty, \omega}$ is \emph{universal} if it is of the form $\forall x_0 \forall x_1 \ldots \forall x_n \phi (x_0, x_1, \ldots, x_n)$, where $\phi$ is a quantifier-free $\Ll_{\infty, \omega}$ formula. An $\Ll_{\infty, \omega}$-theory is \emph{universal} if it consists only of universal $\Ll_{\infty, \omega}$ formulas.
\end{defin}

The following is essentially due to Tarski \cite{tarski-th-models-i}. Only ``(\ref{univ-charact-2}) implies (\ref{univ-charact-1})'' will be used. Tarski proved the result for $\Ll_{\omega, \omega}$, so we sketch a proof of the $\Ll_{\infty, \omega}$ case for the convenience of the reader.

\begin{fact}\label{univ-charact}
  Let $K$ be a class of structures in a fixed vocabulary $\tau = \tau (K)$. Set $\lambda := |\tau| + \aleph_0$. The following are equivalent.
  
  \begin{enumerate}
  \item\label{univ-charact-1} $K$ is a universal class.
  \item\label{univ-charact-2} $K = \operatorname{Mod} (T)$, for some universal $\Ll_{\lambda^+, \omega}$ theory $T$ with $|T| \le \lambda$.
  \end{enumerate}
\end{fact}
\begin{proof}[Proof sketch]
  (\ref{univ-charact-2}) implies (\ref{univ-charact-1}) is straightforward. We show (\ref{univ-charact-1}) implies (\ref{univ-charact-2}). Note that for any fixed finitely generated $\tau$-structure $M$, the class $K_{\neg M}$ of $\tau$-structures that do not contain (as a substructure) a copy of $M$ is axiomatized by a universal $\Ll_{\lambda^+, \omega}$-sentence. Further, there are only $\lambda$-many isomorphism types of finitely generated $\tau$-structures.

  Now for any universal class $K$ in the vocabulary $\tau$, let $\Gamma$ be the class of finitely generated $\tau$-structures that are not contained in any member of $K$. With a directed system argument, one sees that $K$ is exactly the class of $\tau$-structures that do not contain a copy of a member of $\Gamma$.
\end{proof}
\begin{remark}
  Fact \ref{univ-charact} shows that $K$ is axiomatized by a single $\Ll_{\lambda^+, \omega}$-sentence (take the conjunctions of all the formulas in $T$). However it need \emph{not} be true that $K$ is axiomatized by a single \emph{universal} $\Ll_{\lambda^+, \omega}$-sentence: consider the class of directed graphs that do not contain a finite cycle. Confusingly, Malitz \cite{malitz-universal} calls a sentence universal (we will say it is \emph{Malitz-universal}) if it has no existential quantifiers and negations are only applied to atomic formulas. Thus the class of directed graphs without finite cycles is axiomatizable by a single Malitz-universal sentence but not by a single universal sentence. Even worse, the class of all finite sets is axiomatizable by a single Malitz-universal sentence but is not a universal class (it is not closed under unions).
\end{remark}

Universal classes are abstract elementary classes:

\begin{defin}[Definition 1.2 in \cite{sh88}]\label{aec-def}
  An \emph{abstract elementary class} (AEC for short) is a pair $\K = (K, \lea)$, where:

  \begin{enumerate}
    \item $K$ is a class of $\tau$-structures, for some fixed vocabulary $\tau = \tau (\K)$. 
    \item $\lea$ is a partial order (that is, a reflexive and transitive relation) on $K$. 
    \item $(K, \lea)$ respects isomorphisms: If $M \lea N$ are in $K$ and $f: N \cong N'$, then $f[M] \lea N'$. In particular (taking $M = N$), $K$ is closed under isomorphisms.
    \item If $M \lea N$, then $M \subseteq N$. 
    \item Coherence: If $M_0, M_1, M_2 \in K$ satisfy $M_0 \lea M_2$, $M_1 \lea M_2$, and $M_0 \subseteq M_1$, then $M_0 \lea M_1$;
    \item Tarski-Vaught axioms: Suppose $\delta$ is a limit ordinal and $\seq{M_i \in K : i < \delta}$ is an increasing chain. Then:

        \begin{enumerate}

            \item $M_\delta := \bigcup_{i < \delta} M_i \in K$ and $M_0 \lea M_\delta$.
            \item\label{smoothness-axiom}Smoothness: If there is some $N \in K$ so that for all $i < \delta$ we have $M_i \lea N$, then we also have $M_\delta \lea N$.

        \end{enumerate}

    \item Löwenheim-Skolem-Tarski axiom: There exists a cardinal $\lambda \ge |\tau(\K)| + \aleph_0$ such that for any $M \in K$ and $A \subseteq |M|$, there is some $M_0 \lea M$ such that $A \subseteq |M_0|$ and $\|M_0\| \le |A| + \lambda$. We write $\LS (\K)$ for the minimal such cardinal.
\end{enumerate}
\end{defin}
\begin{remark} \
  \begin{enumerate}
    \item When we write $M \lea N$, we implicitly also mean that $M, N \in K$.
    \item We write $\K$ for the pair $(K, \lea)$, and $K$ (no boldface) for the actual class. However we may abuse notation and write for example $M \in \K$ instead of $M \in K$ when there is no danger of confusion. Note that in this paper we will sometimes work with two AECs $\K^1$, $\K^2$ that happen to have the same underlying class $K$ but not the same ordering. 
  \end{enumerate}
\end{remark}

Notice that if $K$ is a universal class, then $\K := (K, \subseteq)$ is an AEC with $\LS (\K) = |\tau (K)| + \aleph_0$. Throughout this paper we will use the following notation:

\begin{notation}\label{univ-notation}
  Let $K$ be a universal class. We think of $K$ as the AEC $\K := (K, \subseteq)$, and may write ``$\K$ is a universal class'' instead of ``$K$ is a universal class''.
\end{notation}

We will also have to deal with AECs that may not satisfy the smoothness axiom:

\begin{defin}[I.1.2.(2) in \cite{shelahaecbook}]
  A \emph{weak AEC} is a pair $\K = (K, \lea)$ satisfying all the axioms of AECs except perhaps smoothness ((\ref{smoothness-axiom}) in Definition \ref{aec-def}).
\end{defin}

Shelah introduced the following parametrized version of smoothness:

\begin{defin}[V.1.18.(3) in \cite{shelahaecbook2}]\label{smooth-def}
  Let $\K$ be a weak AEC. Let $\lambda \ge \LS (\K)$ and let $\delta$ be a limit ordinal. We say that $\K$ is \emph{$(\le \lambda, \delta)$-smooth} if for any increasing chain $\seq{M_i : i \le \delta}$ with $\|M_i\| \le \lambda$ for all $i < \delta$ and $\|M_\delta\| \le \lambda + \delta$, we have that $\bigcup_{i < \delta} M_i \lea M_{\delta}$. $(\le \lambda, \le \kappa)$-smooth means $(\le \lambda, \delta)$-smooth for all $\delta \le \kappa$, and similarly for the other variations.
\end{defin}
\begin{remark}\label{smooth-equiv-rmk}
  Above, we could have allowed $\|M_{\delta}\| > \lambda + \delta$ and gotten an equivalent definition. Indeed, if $M_i \lea M_\delta$ for all $i < \delta$ and we want to see that $\bigcup_{i < \delta} M_i \lea M_\delta$, we can use the Löwenheim-Skolem-Tarski axiom to take $N \lea M_{\delta}$ containing $\bigcup_{i < \delta} |M_i|$ and having size at most $\lambda + \delta$. Then we can use coherence to see that $M_i \lea N$ for all $i < \delta$, hence by smoothness, $\bigcup_{i < \delta} M_i \lea N$ and so by transitivity of $\lea$, $\bigcup_{i < \delta} M_i \lea M_\delta$.
\end{remark}

We now list a several known facts about AECs that we will use. First, recall that an AEC $\K$ is determined by its restriction $\K_{\LS (\K)}$ to models of size $\LS (\K)$. More precisely:

\begin{fact}[II.1.23 in \cite{shelahaecbook}]\label{aec-bottom-det}
  Assume $\K^1$ and $\K^2$ are AECs with $\lambda := \LS (\K^1) = \LS (\K^2)$. If $\K_\lambda^1 = \K_\lambda^2$ (so also $\leap{\K^1}$ and $\leap{\K^2}$ coincide on the models of size $\lambda$), then $\K_{\ge \lambda}^1 = \K_{\ge \lambda}^2$.
\end{fact}

We will use the relationship between the ordering of any AEC and elementary equivalence in a sufficiently powerful infinitary logic:

\begin{fact}\label{elem-facts}
  Let $\K$ be an AEC and let $M, N \in \K$.

  \begin{enumerate}
    \item\cite[Theorem 7.2.(b)]{kueker2008}\label{elem-facts-1} If $M \lee_{\Ll_{\infty, \LS (\K)^+}} N$, then $M \lea N$.
    \item\cite[IV.1.12.(2)]{shelahaecbook}\label{elem-facts-2} Let $\lambda$ be an infinite cardinal such that $\K$ is categorical in $\lambda$ and $\lambda = \lambda^{\LS (\K)}$. If $M, N \in \K_{\lambda}$ and $M \lea N$, then $M \lee_{\Ll_{\infty, \LS (\K)^+}} N$.
  \end{enumerate}
\end{fact}
\begin{remark}
  Shelah's proof of Fact \ref{elem-facts}.(\ref{elem-facts-2}) is short and elementary but in \cite[Section IV.2]{shelahaecbook}, he attempts to remove the ``$\lambda = \lambda^{\LS (\K)}$'' restriction. We rely on parts of Shelah's argument to get amalgamation in \cite{ap-universal-v9} (e.g.\ in the proof of Fact \ref{univ-fact}.(\ref{univ-fact-1})), but in this paper we have a different strategy to get amalgamation and hence do not need to rely on the deep results from \cite[Chapter IV]{shelahaecbook}.
\end{remark}

We will also use that AECs have a Hanf number. Below, we write $\delta (\lambda)$ for the pinning down ordinal at $\lambda$: the first ordinal that is not definable in $\Ll_{\lambda^+, \omega}$. We will also deal with the more general $\delta (\lambda, \kappa)$ (the least ordinal not definable using a $\operatorname{PC}_{\lambda, \kappa}$ class, see \cite[VII.5.5.1]{shelahfobook} for a precise definition). Recall the following well-known facts about this ordinal (see e.g.\ \cite[VII.5]{shelahfobook}): 

\begin{fact}\label{delta-facts} \
  \begin{enumerate}
    \item (Lopez-Escobar) $\delta (\aleph_0) = \omega_1$.
    \item (Morley and C.C. Chang) For any infinite cardinals $\lambda$ and $\kappa$, $\delta (\lambda, \kappa) \le (2^{\lambda})^+$.
  \end{enumerate}
\end{fact}

\begin{defin}\label{hanf-aec-def}
  Let $\K$ be an AEC.

  \begin{enumerate}
    \item Let $\lambda (\K)$ be the least cardinal $\lambda \ge \LS (\K)$ such that there exists a vocabulary $\tau_1 \supseteq \tau (\K)$, a first-order $\tau_1$-theory $T_1$, and a set of $T_1$-types $\Gamma$ such that:
      \begin{enumerate}
        \item\label{k-cond} $\K = \text{PC} (T_1, \Gamma, \tau (\K))$.
        \item\label{order-cond} For $M, N \in \text{EC} (T_1, \Gamma)$, if $M \subseteq N$, then $M \rest \tau (\K) \lea N \rest \tau (\K)$.
        \item $|T_1| + |\tau_1| \le \LS (\K)$ and $|\Gamma| \le \lambda$.
      \end{enumerate}
    \item Let $\delta (\K) := \delta (\LS (\K), \lambda (\K))$.
    \item Let $\hanf{\K} := \beth_{\delta (\K)}$.
  \end{enumerate}
\end{defin}
\begin{remark}\label{univ-delta}
  By Chang's presentation theorem \cite{chang-pres}, if $\K$ is axiomatized by an $\Ll_{\lambda^+, \omega}$ sentence, and the ordering is just substructure (as for universal classes), then $\lambda (\K) \le \lambda$. In particular (see Fact \ref{univ-charact}) $\lambda (\K) = |\tau (\K)| + \aleph_0$ for any universal class $\K$.
\end{remark}

It makes sense to talk of $\lambda (\K)$ because of Shelah's presentation theorem:

\begin{fact}[I.1.9 in \cite{shelahaecbook}]\label{pres-thm}
  For any AEC $\K$, there exists a vocabulary $\tau_1 \supseteq \tau (\K)$, a first-order $\tau_1$-theory $T_1$, and a set of $T_1$-types $\Gamma$ such that  (\ref{k-cond}) and (\ref{order-cond}) in Definition \ref{hanf-aec-def} hold and $|T_1| + |\tau_1| \le \LS (\K)$. Thus $\lambda (\K) \le 2^{\LS (\K)}$.
\end{fact}

\begin{defin}\label{hanf-def}
  For an infinite cardinal $\lambda$, let $\hanf{\lambda} := \beth_{(2^{\lambda})^+}$. 
\end{defin}
\begin{remark}\label{hanf-bound}
  By Facts \ref{delta-facts} and \ref{pres-thm}, For any AEC $\K$, $\hanf{\K} \le \hanf{\LS (\K)}$.
\end{remark}

The reason $\hanf{\K}$ is interesting is because it is a Hanf number for $\K$ (this follows from Chang's result on the Hanf number of PC classes \cite{chang-pres}).

\begin{fact}\label{hanf-arb-large}
  Let $\K$ be an AEC. If $\K$ has a model of size $\hanf{\K}$, then $\K$ has arbitrarily large models. 
\end{fact}

In the rest of this section, we quote categoricity transfer results that we will use. We assume that the reader is familiar with notions such as amalgamation, joint embedding, Galois types, Ehrenfeucht-Mostowski models, and tameness (see for example \cite{baldwinbook09}). The notation we use is standard and is described in details at the beginning of \cite{sv-infinitary-stability-afml} (for Ehrenfeucht-Mostowski models, we use the notation in \cite[IV.0.8]{shelahaecbook}\footnote{For $\K$ an AEC, we call $\Phi$ an \emph{EM blueprint for $\K$} if $\Phi \in \Upsilon_{\K}^{\text{or}}$.}). For example, we write $\gtp (\bb / M; N)$ for the Galois type of $\bb$ over $M$, as computed in $N$. This assumes that we are working inside an AEC $\K$ that is clear from context. When we want to emphasize $\K$, we will write $\gtp_{\K} (\bb / M; N)$.

The following result is implicit in the proof of \cite[Corollary 4.3]{tamenesstwo}. For completeness, we sketch a proof.

\begin{fact}\label{categ-facts-ap}
  If $\K$ is an AEC with amalgamation and arbitrarily large models, then the categoricity spectrum (i.e.\ the class of cardinals $\lambda \ge \LS (\K)$ such that $\K$ is categorical in $\lambda$) is closed. That is, if $\lambda > \LS (\K)$ is a limit cardinal and $\K$ is categorical in unboundedly many cardinals below $\lambda$, then $\K$ is also categorical in $\lambda$.
\end{fact}
\begin{proof}
  Let $\lambda > \LS (\K)$ be a limit cardinal such that $\K$ is categorical in unboundedly many cardinals below $\lambda$. We show that $\K$ is categorical in $\lambda$. We proceed in several steps:

  \begin{enumerate}
  \item $\K$ is (Galois) stable in every $\mu \in [\LS (\K), \lambda)$.
    [Why? Pick $\mu' \in (\mu, \lambda)$ such that $\K$ is categorical in $\mu'$. Since $\K$ has arbitrarily large models, we can use Ehrenfeucht-Mostowski models and the standard argument of Morley (see e.g.\ the proof of \cite[Claim I.1.7]{sh394-updated}) to see that $\K$ is stable in $\mu$.].
  \item For every categoricity cardinal $\mu \in (\LS (\K), \lambda)$, the model of size $\mu$ is (Galois) saturated.
    [Why? Using stability we can build a $\mu_0$-saturated model of size $\mu$ for every $\mu_0 \in (\LS (\K), \mu)$, and then use categoricity.]
  \item Every model of size $\lambda$ is saturated.
    [Why? Let $M \in \K_\lambda$. Let $N \in \K_{<\lambda}$ be such that $N \lea M$. Let $p \in \gS (N)$. Let $\mu := \|N\|$ and let $\mu' \in (\mu, \lambda)$ be a categoricity cardinal. Let $N' \in \K_{\mu'}$ be such that $N \lea N' \lea M$. By the previous step, $N'$ is saturated, and therefore realizes $p$. Since $N' \lea M$, $M$ also realizes $p$.
    \item $\K$ is categorical in $\lambda$.
      [Why? By uniqueness of saturated models.]
  \end{enumerate}
\end{proof}

To state the next categoricity transfer, we first recall Shelah's notion of an AEC having primes. The intuition is that the AEC has prime models over every set of the form $M \cup \{a\}$, for $M \in \K$. This is described formally using Galois types.

\begin{defin}[III.3.2 in \cite{shelahaecbook}]
  Let $\K$ be an AEC.

  \begin{enumerate}
    \item $(a, M, N)$ is a \emph{prime triple} if $M \lea N$, $a \in |N| \backslash |M|$, and for every $N' \in \K$, $a' \in |N'|$, such that $\gtp (a / M; N) = \gtp (a' / M; N')$, there exists $f: N \xrightarrow[M]{} N'$ with $f (a) = a'$.
    \item $\K$ \emph{has primes} if for any nonalgebraic Galois type $p \in \gS (M)$ there exists a prime triple $(a, M, N)$ such that $p = \gtp (a / M; N)$.
  \end{enumerate}
\end{defin}

By taking the closure of the relevant set under the functions of an ambient model, we obtain:

\begin{fact}[Remark 5.3 in \cite{ap-universal-v9}]\label{univ-prime}
  Any universal class $\K = (K, \subseteq)$ has primes.
\end{fact}
\begin{remark}\label{prime-loss-rmk}
  Having primes is a property of the AEC $\K = (K, \lea)$, not just of the class $K$. Thus even though for any universal class $K$, $(K, \subseteq)$ has primes, changing the order may lead to an AEC $(K, \lea)$ that may \emph{not} have primes anymore.
\end{remark}

The following is a ZFC approximation of Shelah's eventual categoricity conjecture in tame AECs with amalgamation. It combines works of Makkai-Shelah \cite{makkaishelah}, Shelah \cite{sh394}, Grossberg and VanDieren \cite{tamenesstwo, tamenessthree}, and the author \cite{ap-universal-v9, categ-primes-v4, downward-categ-tame-v6-toappear}.

\begin{fact}\label{categ-facts}
  Let $\K$ be a $\LS (\K)$-tame AEC with amalgamation and arbitrarily large models. Let $\lambda > \LS (\K)$ be such that $\K$ is categorical in $\lambda$.

  \begin{enumerate}
    \item\label{categ-facts-1} \cite[Theorem 9.8]{downward-categ-tame-v6-toappear}\footnote{The version for classes of models axiomatized by an $\Ll_{\kappa, \omega}$ theory, $\kappa$ strongly compact, appears in \cite{makkaishelah}. It generalizes to AECs with amalgamation when the model in the categoricity cardinal is saturated (see \cite[Lemma II.1.6]{sh394} or \cite[Theorem 14.8]{baldwinbook09}). In the tame case, the model in the categoricity cardinal is always saturated (by the Shelah-Villaveces theorem \cite[Theorem 2.2.1]{shvi635} together with the upward superstability transfer of the author \cite[Proposition 10.10]{indep-aec-apal}). In all these arguments, it seems that the amalgamation property is used in a strong way.} If $\delta$ is a limit ordinal that is divisible by $\left(2^{\LS (\K)}\right)^+$, then $\K$ is categorical in $\beth_{\delta}$.
    \item\label{categ-facts-2} $\K$ is categorical in all $\lambda' \ge \min (\lambda, \hanf{\LS (\K)})$ when at least one of the following holds:
      \begin{enumerate}
        \item\label{categ-facts-2a} \cite[10.3, 10.6]{downward-categ-tame-v6-toappear}\footnote{The \emph{upward} part of this transfer (i.e.\ concluding categoricity in all $\mu' \ge \mu$ is due to Grossberg and VanDieren \cite{tamenessthree}).} There exists a \emph{successor} cardinal $\mu > \LS (\K)^+$ such that $\K$ is categorical in $\mu$.
        \item\label{categ-facts-2b} \cite[Theorem 10.9]{downward-categ-tame-v6-toappear}\footnote{The main ideas of the transfer with primes appear in \cite{ap-universal-v9, categ-primes-v4} but there the threshold is higher (around $\beth_{\hanf{\LS (\K)}}$). The improved threshold of $\hanf{\LS (\K)}$ can be obtained from Fact \ref{categ-facts}.(\ref{categ-facts-2a}).} $\K$ has primes.
      \end{enumerate}
  \end{enumerate}
\end{fact}
\begin{remark}
  In Fact \ref{categ-facts}, we do not use that $\K$ has joint embedding: we can find a sub-AEC $\K^0$ of $\K$ that has joint embedding and work within $\K^0$. See Definition \ref{k-ast-def}.
\end{remark}
\begin{remark}
  If in Fact \ref{categ-facts} we start instead with a $\chi$-tame AEC (with $\chi > \LS (\K)$), the same conclusions hold for $\K_{\ge \chi}$.
\end{remark}

\section{Compatible pairs of AECs}\label{compat-sec}

Let $K$ be a universal class. A central result of Shelah \cite[V.B]{shelahaecbook2} is that if $K$ does not have the order property, there is an ordering $\le$ such that $(K, \le)$ has several structural properties, including amalgamation. The downside is that $(K, \le)$ might loose the smoothness axiom, i.e.\ it may only be a weak AEC. We will give the precise statement of Shelah's result and discuss its implications in the next sections. 

Here, we look at the situation abstractly: we consider pairs of weak AECs $\K^1 = (K^1, \leap{\K^1})$ and $\K^2 = (K^2, \leap{\K^2})$ satisfying a compatibility condition. The case of interest is $\K^1 = (K, \subseteq)$ and $\K^2 = (K, \le)$.

\begin{defin}\label{compatible-def}
  For $\ell = 1,2$, let $\K^\ell = (K^\ell, \leap{\K^\ell})$ be weak AECs. $\K^1$ and $\K^2$ are \emph{compatible} if:

  \begin{enumerate}
    \item $\tau(\K^1) = \tau(\K^2)$.
    \item\label{cond-2} For any $\lambda > \LS (\K^1) + \LS (\K^2)$, if either $\K^1$ or $\K^2$ is categorical in $\lambda$, then $K^1_\lambda = K^2_{\lambda}$.
  \end{enumerate}

  We write $\LS (\K^1, \K^2)$ instead of $\LS (\K^1) + \LS (\K^2)$.
\end{defin}
\begin{remark}
  This definition is really only useful when one of the classes is categorical. Note that in (\ref{cond-2}), we only ask for $K^1_\lambda = K^2_\lambda$, i.e.\ the isomorphism type of the model of size $\lambda$ must be the same in both classes, but the orderings need not agree.
\end{remark}

For the rest of this section, we assume (and will emphasize the compatibility hypothesis again):

\begin{hypothesis}
  $\K^1 = (K^1, \leap{\K^1})$ and $\K^2 = (K^2, \leap{\K^2})$ are \emph{compatible} weak AECs. We set $\tau := \tau (\K^1) = \tau (\K^2)$.
\end{hypothesis}

Assume that $\K^1$ is categorical in a $\lambda > \LS (\K^1, \K^2)$. What can we say about $\K^2$? If $\K^1$ is a universal class and $\K^2$ is as above, $\K^1$ is an AEC, and one of our ultimate goal is to show that $\K^2$ is also an AEC. The following result will turn out to be key. Under some assumptions, $\K^2$ is stable below the categoricity cardinal.

\begin{lem}\label{stable-prop}
  Assume:
  \begin{enumerate}
    \item $\K^1$ is an AEC with arbitrarily large models.
    \item $\K^2$ has amalgamation and joint embedding.
    \item $\K^1$ and $\K^2$ are compatible.
  \end{enumerate}

  Let $\lambda > \LS (\K^1, \K^2)$. If $\K^2$ (and so by compatibility also $\K^1$) is categorical in $\lambda$, then $\K^2$ is $(<\omega)$-stable in all $\mu \in [\LS (\K^1, \K^2), \lambda)$ such that $\mu^+ < \lambda$. That is, for any such $\mu$ and any $M \in \K_\mu^2$, $|\gS^{<\omega}_{\K^2} (M)| \le \mu$
\end{lem}

Before starting the proof, a few comments are in order. First note that the case $\K^1 = \K^2$ is a classical result that can be traced back to Morley \cite[Theorem 3.7]{morley-cip}. It appears explicitly as \cite[Claim I.1.7]{sh394}. The proof uses Ehrenfeucht-Mostowski (EM) models. Here, we have additional difficulties since the EM models are well-behaved really only for $\K^1$ and not for $\K^2$ (in fact, $\K^2$ may be only a weak AEC, so may not have any suitable EM blueprint). More precisely, if $\Phi$ is an EM blueprint for $\K^1$ and $I \subseteq J$ are linear orders, then $\EM_\tau (I, \Phi) \leap{\K^1} \EM_\tau (J, \Phi)$ but possibly $\EM_\tau (I, \Phi) \not \leap{\K^2} \EM_\tau (J, \Phi)$. Thus a Galois type of $\K^2$ computed inside $\EM_\tau (I, \Phi)$ may not be the same as one computed in $\EM_\tau (J, \Phi)$. For this reason, we want to use only that Galois types are invariant under isomorphisms in the proof, and hence want to use the existence of certain linear orderings with many automorphisms. 

Fortunately, Shelah gives a proof of the case $\K^1 = \K^2$ in \cite[Claim I.1.7]{sh394-updated} (the online version of \cite{sh394}) that we can imitate. It uses the following fact:

\begin{fact}[IV.5.1.(2) in \cite{shelahaecbook}]\label{nice-order}
  Let $\theta < \lambda$ be infinite cardinals with $\theta$ regular. There exists a linear order $I$ of size $\lambda$ such that for every $I_0 \subseteq I$ of size less than $\theta$, there is $J \subseteq I$ such that:

  \begin{enumerate}
    \item $I_0 \subseteq J$.
    \item $\|J\| = \|I_0\| + \aleph_0$.
    \item For any $\ba \in \fct{<\omega}{I}$, there is $f \in \Aut_{I_0} (I)$ such that $f (\ba) \in \fct{<\omega}{J}$.
  \end{enumerate}
\end{fact}

\begin{proof}[Proof of Lemma \ref{stable-prop}]
  Since $\K^1$ has arbitrarily large models and is an AEC, it has an Ehrenfeucht-Mostowski blueprint $\Phi$. Let $\mu \in [\LS (\K^1, \K^2), \lambda)$ and let $M \in K_{\mu}^2$. We want to see that $|\gS_{\K^2} (M)| \le \mu$. Let $I$ be as described by Fact \ref{nice-order} (where $\theta$ there stands for $\mu^+$ here, we are using that $\mu^+ < \lambda$). Suppose for a contradiction that $|\gS_{\K^2}^{<\omega} (M)| > \mu$. Then using amalgamation we can find $N \in K^2$ with $M \leap{\K^2} N$ and a sequence $\seq{\ba_i \in \fct{<\omega}{|N|} : i < \mu^+}$ such that for $i < j < \mu^+$, $\gtp_{\K^2} (\ba_i / M; N) \neq \gtp_{\K^2} (\ba_j / M; N)$.

By joint embedding and categoricity, without loss of generality $N = \EM_\tau (I, \Phi)$. Now let $I_0 \subseteq I$ be such that $|I_0| = \mu$ and $M \subseteq \EM_{\tau} (I_0, \Phi)$. Let $J$ be as given by the definition of $I$ and let $M_1 := \EM_{\tau} (J, \Phi)$. We have that for each $i < \mu^+$, there is a finite linear order $I_i \subseteq I$ generating $\ba_i$, so pick $f_i \in \Aut_{I_0} (I)$ such that $f_i [I_i] \subseteq J$. Let $\widehat{f_i} \in \Aut_{M_1} (N)$ be the automorphism of $N = \EM_\tau (I, \Phi)$ naturally induced by $f_i$. Then $\widehat{f_i} (\ba_i) \in |M_1|$. By the pigeonhole principle, without loss of generality there is $\bb \in |M_1|$ such that for all $i < \mu^+$, $\widehat{f_i} (\ba_i) = \bb$. But this means that for $i < \mu^+$:

$$\gtp_{\K^2} (\ba_i / M; N) = \gtp_{\K^2} (\widehat{f_i} (\ba_i) / M; N) = \gtp_{\K^2} (\bb / M; N)$$

So for $i < j < \mu^+$, $\gtp_{\K^2} (\ba_i / M; N) = \gtp_{\K^2} (\ba_j / M; N)$, a contradiction.
\end{proof}
\begin{remark}\label{omega-stability-rmk}
  We emphasize that Lemma \ref{stable-prop} establishes stability for all finite types and not just stability for types of length one (in the framework of weak AECs we do not know if the two notions are the same). This slightly stronger statement will be used in the proof of Theorem \ref{universal-class-structure}. There we want to derive a contradiction with Theorem \ref{smoothness-unstable}, which only concludes unstability for finite types, not unstability for types of length one.
\end{remark}

For the rest of this section, we assume that $\K^1$ and $\K^2$ are both AECs and discuss categoricity transfers (generalizing Fact \ref{categ-facts}) to this setup. First, we show that categoricity in a suitable cardinal implies that the two classes (and their ordering) are equal on a tail.

\begin{lem}\label{equal-aecs-prop}
  Assume $\K^1$ and $\K^2$ are compatible AECs. Let $\lambda$ be an infinite cardinal such that:
  \begin{enumerate}
\item  $\K^1$ is categorical in $\lambda$.
\item $\lambda = \lambda^{\LS (\K^1, \K^2)}$.
  \end{enumerate}

  Then $\K_{\ge \lambda}^1 = \K_{\ge \lambda}^2$ (so also the orderings are equal).
\end{lem}
\begin{proof}
  By compatibility, $K_\lambda^1 = K_\lambda^2$. By Fact \ref{aec-bottom-det} (where $\K^1, \K^2$ there stand for $\K_{\ge \lambda}^1$, $\K_{\ge \lambda}^2$ here), it is enough to show that the orderings of $\K^1$ and $\K^2$ coincide on $K_\lambda^1$. So let $M, N \in K_{\lambda}^1$. We show that $M \leap{\K^1} N$ implies $M \leap{\K^2} N$ (the converse is symmetric).

  So assume that $M \leap{\K^1} N$. By Fact \ref{elem-facts}.(\ref{elem-facts-2}) (where $\K$, $\lambda$ there stand for $\K_{\ge \LS (\K^1, \K^2)}^1$, $\lambda$ here), $M \lee_{\Ll_{\infty, \LS (\K^1, \K^2)^+}} N$. By Fact \ref{elem-facts}.(\ref{elem-facts-1}) (where $\K$ there stands for $\K_{\ge \LS (\K^1, \K^2)}^2$ here), $M \leap{\K^2} N$, as desired.
\end{proof}

The next result shows that if one of the classes has amalgamation, we can find a categoricity cardinal satisfying the condition of the previous lemma.

\begin{thm}\label{ap-categ-equal}
  Assume $\K^1$ and $\K^2$ are compatible AECs categorical in a proper class of cardinals. If $\K^1$ has amalgamation, then there exists $\lambda$ such that $\K_{\ge \lambda}^1 = \K_{\ge \lambda}^2$ (so also the orderings are equal).
\end{thm}
\begin{proof}
  Because $\K^1$ is categorical in a proper class of cardinals, it has arbitrarily large models, so by Fact \ref{categ-facts-ap}, $\K^1$ is categorical on a closed unbounded class of cardinals. In particular, one can find an infinite cardinal $\lambda$ such that $\K^1$ is categorical in $\lambda$ and $\lambda = \lambda^{\LS (\K^1, \K^2)}$. By Lemma \ref{equal-aecs-prop}, $\K_{\ge \lambda}^1 = \K_{\ge \lambda}^2$.
\end{proof}

We end this section with a categoricity transfer. Intuitively, this shows that if we start with an AEC $\K^1$ with primes, it is enough to change its ordering (getting an AEC $\K^2$) so that $\K^2$ has amalgamation and is tame (it may lose existence of primes, see Remark \ref{prime-loss-rmk}). This is especially relevant to universal classes, since they always have primes (Fact \ref{univ-prime}). Note that Fact \ref{categ-facts}.(\ref{categ-facts-2b}) is the case $\K^1 = \K^2$.

\begin{thm}\label{abstract-compatible-categ}
  Assume $\K^1$ and $\K^2$ are compatible AECs such that:

  \begin{enumerate}
    \item $\K^1$ has primes.
    \item $\K^2$ has amalgamation, arbitrarily large models, and is $\LS (\K^2)$-tame.
  \end{enumerate}

  If $\K^2$ is categorical in a $\lambda > \LS (\K^2)$, then $\K^2$ is categorical in all $\lambda' \ge \min (\lambda, \hanf{\LS (\K^2)})$.
\end{thm}
\begin{proof}
  By Fact \ref{categ-facts}.(\ref{categ-facts-1}), $\K^2$ is categorical in a proper class of cardinals. By Theorem \ref{ap-categ-equal} (where the role of $\K^1$ and $\K^2$ is switched), we can fix a cardinal $\lambda_0$ such that $\K^1_{\ge \lambda_0} = \K^2_{\ge \lambda_0}$. In particular, their orderings also coincide and so $\K^2_{\ge \lambda_0}$ has primes. By Fact \ref{categ-facts}.(\ref{categ-facts-2b}), $\K^2_{\ge \lambda_0}$ is categorical on a tail, and in particular in a successor cardinal. Applying Fact \ref{categ-facts}.(\ref{categ-facts-2a}) to $\K^2$, this implies that $\K^2$ is categorical in all $\lambda' \ge \min (\lambda, \hanf{\LS (\K^2)}$, as desired.
\end{proof}

\section{Independence in weak AECs}\label{axfr-sec}

$\Axfr$ is an axiomatic framework for independence in weak AECs that Shelah introduces in \cite{sh300-orig}. The main motivation for the axioms is that if $K$ is a universal class that does not have the order property, then there is an ordering $\le$ such that $(K, \le)$ satisfies $\Axfr$ (see Section \ref{uc-structure-sec}). Here, we repeat the definition and state some facts that we will use. We quote from Chapter V of \cite{shelahaecbook2}, an updated version of \cite{sh300-orig}.

\begin{defin}[$\Axfr$, V.B in \cite{shelahaecbook2}]\label{axfr-def}
  $(\K, \nf, \cl)$ satisfies\footnote{In order to be consistent with \cite{ap-universal-v9}, we write $\cl$ rather than Shelah's $\langle \rangle_{\text{gn}}$.} $\Axfr$ if:

  \begin{enumerate}
    \item $\K$ is a weak AEC.
    \item For each $N \in \K$, $\cl^N$ is a function from $\mathcal{P}(|N|)$ to $\mathcal{P}(|N|)$. Often, $\cl^N (A)$ induces a $\tau (\K)$-substructure $M$ of $N$. In this case, we identify $\cl^N (A)$ with $M$. We require $\cl$ to satisfy the following axioms: For $N, N' \in \K$, $A, B \subseteq |N|$:
      \begin{enumerate}
        \item Invariance: If $f: N \cong N'$, then $\cl^{N'} (f[A]) = f[\cl^N (A)]$.
        \item Monotonicity 1: If $A \subseteq B$, then $\cl^N (A) \subseteq \cl^N (B)$.
        \item\label{axfr-def-mon2} Monotonicity 2: If $N \lea N'$, then $\cl^N (A) = \cl^{N'} (A)$.
        \item Idempotence: $\cl^N (\cl^N (A)) = \cl^N (A)$.
      \end{enumerate}
    \item $\nf$ is a 4-ary relation on $\K$. We write $\nfs{M_0}{M_1}{M_2}{M_3}$ instead of $\nf (M_0, M_1, M_2, M_3)$. We require that $\nf$ satisfies the following axioms:

      \begin{enumerate}
        \item $\nfs{M_0}{M_1}{M_2}{M_3}$ implies that for $\ell = 1,2$, $M_0 \lea M_\ell \lea M_3$.
        \item Invariance: If $f: M_3 \cong M_3'$ and $\nfs{M_0}{M_1}{M_2}{M_3}$, then $\nfs{f[M_0]}{f[M_1]}{f[M_2]}{M_3'}$.
        \item Monotonicity 1: If $\nfs{M_0}{M_1}{M_2}{M_3}$ and $M_3 \lea M_3'$, then $\nfs{M_0}{M_1}{M_2}{M_3'}$.
        \item Monotonicity 2: If $\nfs{M_0}{M_1}{M_2}{M_3}$ and $M_0 \lea M_2' \lea M_2$, then $\nfs{M_0}{M_1}{M_2'}{M_3}$.
        \item Base enlargement: If $\nfs{M_0}{M_1}{M_2}{M_3}$ and $M_0 \lea M_2' \lea M_2$, then $\nfs{M_2'}{\cl^{M_3} (M_2' \cup M_1)}{M_2}{M_3}$.
        \item Symmetry: If $\nfs{M_0}{M_1}{M_2}{M_3}$, then $\nfs{M_0}{M_2}{M_1}{M_3}$.
        \item\label{axfr-def-existence} Existence: If $M_0 \lea M_\ell$, $\ell = 1,2$, then there exists $N \in \K$ and $f_\ell : M_\ell \xrightarrow[M_0]{} N$, $\ell = 1,2$, such that $\nfs{M_0}{f[M_1]}{f[M_2]}{N}$.
        \item Uniqueness: If for $\ell = 1,2$, $\nfs{M_0^\ell}{M_1^\ell}{M_2^\ell}{M_3^\ell}$ and $for i < 3$, $f_i: M_i^1 \cong M_i^2$ are such that $f_0 \subseteq f_1$, $f_0 \subseteq f_2$, then there exists $N \in \K$ with $M_3^2 \lea N$ and $h: M_3^1 \rightarrow N$ such that $f_1 \cup f_2 \subseteq h$.
        \item Finite character: If $\delta$ is a limit ordinal, $\seq{M_{2, i} : i \le \delta}$ is increasing and continuous, $M_0 \lea M_{1, 0}$, and $\nfs{M_0}{M_1}{M_{2, \delta}}{M_3}$, then $\cl^{M_3} (M_1 \cup M_{2, \delta}) = \bigcup_{i < \delta} \cl^{M_3} (M_1 \cup M_{2, i})$.
      \end{enumerate}
  \end{enumerate}

  We say that a weak AEC $\K$ \emph{satisfies $\Axfr$} if there exists $\nf$ and $\cl$ such that $(\K, \nf, \cl)$ satisfies $\Axfr$.
\end{defin}
\begin{remark}
  The definition we give is slightly different from Shelah's: Shelah does not assume that $\K$ has a Löwenheim-Skolem-Tarski number. We do not need the extra generality, although there are places (e.g.\ Section \ref{enum-trees-sec}) where the existence of a Löwenheim-Skolem-Tarski number is not used.
\end{remark}
\begin{remark}
  There is an example (derived from the class of metric graphs, see \cite[V.B.1.22]{shelahaecbook2}) of a triple ($\K, \nf, \cl)$ that satisfies $\Axfr$ but where $\K$ is not an AEC. 
\end{remark}
\begin{remark}\label{ap-rmk}
  If a weak AEC $\K$ satisfies $\Axfr$, then by the existence property for $\nf$, $\K$ has amalgamation.
\end{remark}

In the rest of this section, we assume:

\begin{hypothesis}
  $(\K, \nf, \cl)$ satisfies $\Axfr$.
\end{hypothesis}

The following is easy to see from the definition of the closure operator. 

\begin{fact}\label{cont-of-cl}
  Let $N \in \K$ and let $\seq{A_i : i \in I}$ be a sequence of subsets of $|N|$, $I \neq \emptyset$. Then:

  \begin{enumerate}
    \item $\bigcup_{i \in I} \cl^N (A_i) \subseteq \cl^N (\bigcup_{i \in I} A_i)$.
    \item $\cl^N (\bigcup_{i \in I} A_i) = \cl^N (\bigcup_{i \in I} \cl^N (A_i))$.
  \end{enumerate}
\end{fact}

The following are consequences of the axioms and will all be used in the rest of this paper (as forking calculus tools for Sections \ref{enum-trees-sec} and \ref{uc-structure-sec}).

\begin{fact}\label{basic-axfr-props} \
  \begin{enumerate}
    \item\label{basic-props-1} \cite[V.B.1.21.(1)]{shelahaecbook2} If $\nfs{M_0}{M_1}{M_2}{M_3}$, then $\cl^{M_3} (M_1 \cup M_2) \lea M_3$ and $\nfs{M_0}{M_1}{M_2}{\cl^{M_3} (M_1 \cup M_2)}$.
    \item\label{basic-props-2} \cite[V.C.1.3]{shelahaecbook2} Transitivity: If $\nfs{M_0}{M_1}{M_2}{M_3}$ and $\nfs{M_2}{M_3}{M_4}{M_5}$, then $\nfs{M_0}{M_1}{M_4}{M_5}$.
    \item\label{basic-props-3} \cite[V.C.1.6]{shelahaecbook2} Let $\delta$ be a limit ordinal. Let $\seq{M_i : i \le \delta}$, $\seq{N_i : i \le \delta}$ be $\subseteq$-increasing continuous chains such that for all $i < j < \delta$, $\nfs{M_i}{M_j}{N_i}{N_j}$. Then for all $i \le \delta$, $\nfs{M_i}{M_\delta}{N_i}{N_\delta}$.
    \item\label{basic-props-4} \cite[V.C.1.10.(1)]{shelahaecbook2} Let $\delta$ be a limit ordinal. Let $\seq{M_i : i \le \delta + 1}$, $\seq{N_i^a : i \le \delta}$, $\seq{N_i^b : i \le \delta}$ be increasing continuous chains such that for all $i < \delta$, $\nfs{M_i}{N_i^a}{M_{\delta + 1}}{N_i^b}$ and $N_i^b = \cl^{N_i^b} (M_{\delta + 1} \cup N_i^a)$. Then $\nfs{M_\delta}{N_\delta^a}{M_{\delta + 1}}{N_\delta^b}$.
    \item\label{basic-props-5} Let $\delta$ be a limit ordinal. Let $\seq{M_i : i \le \delta}$, $\seq{N_i : i \le \delta}$ be increasing continuous so that for $i, j < \delta$, $\nfs{M_i}{M_j}{N_i}{N_j}$. Let $M \in \K$ be such that $M_i \lea M$ for all $i < \delta$ (but possibly $M_\delta \not \lea M$). Then there exists $N \in \K$ and an embedding $f: M \xrightarrow[M_\delta]{} N$ such that for all $i < \delta$:
      \begin{enumerate}
        \item $N_i \lea N$.
        \item $\nfs{M_i}{f[M]}{N_i}{N}$.
        \item $N = \cl^N (f[M] \cup N_\delta)$.
      \end{enumerate}
  \end{enumerate}
\end{fact}
\begin{proof}[Proof of (\ref{basic-props-5})]
  This is given by the proof of \cite[V.C.1.11]{shelahaecbook2}, but Shelah omits the end of the proof. We give it here. We build $\seq{N_i^a, N_i^b, f_i : i \le \delta}$ such that:

  \begin{enumerate}
    \item $\seq{N_i^x : i \le \delta}$ is increasing continuous for $x \in \{a, b\}$.
    \item For $i < \delta$, $\nfs{M_i}{M}{N_i^a}{N_i^b}$.
    \item For $i < \delta$, $N_i^b = \cl^{N_i^b} (M \cup N_i^a)$.
    \item For $i \le \delta$,  $f_i : N_i \cong_{M_i} N_i^a$.
  \end{enumerate}

  This is possible by the proof of \cite[V.C.1.11]{shelahaecbook2}. Let us see that it is enough. Find $N \in \K$ and $f: N_\delta^b \cong N$ that extends $f_\delta^{-1}$. We claim that this works. First observe that $f \rest M: M \xrightarrow[M_\delta] N$ as $f$ fixes $M_i$ for each $i < \delta$ and $M \lea N_0^b \lea N_\delta^b$. Now:
  
  \begin{enumerate}
    \item For all $i < \delta$, $N_i \lea N$, since $N_i^a \lea N_i^b \lea N_\delta^b$ and $f_i^{-1} : N_i^a \cong N_i$.
    \item For all $i < \delta$, we have that $\nfs{M_i}{M}{N_i^a}{N_i^b}$ by construction, so applying $f$ to this we get $\nfs{M_i}{f[M]}{f[N_i^a]}{f[N_i^b]}$, i.e.\ $\nfs{M_i}{f[M]}{N_i}{f[N_i^b]}$, so $\nfs{M_i}{f[M]}{N_i}{N}$ by monotonicity.
    \item $N = \cl^N (f[M] \cup N_\delta)$: Why? Note that by continuity $N_\delta^b = \bigcup_{i < \delta} N_i^b$ and the latter is $\bigcup_{i < \delta} \cl^{N_\delta^b} (M \cup N_i^a)$ by construction. Now, $N_\delta^b = \cl^{N_\delta^b} (N_\delta^b) = \cl^{N_\delta^b} (\bigcup_{i < \delta} \cl^{N_\delta^b} (M \cup N_i^a))$. By Fact \ref{cont-of-cl}, this is $\cl^{N_\delta^b} (\bigcup_{i < \delta} M \cup N_i^a) = \cl^{N_\delta^b} (M \cup N_\delta^a)$. We have shown that $N_\delta^b = \cl^{N_\delta^b} (M \cup N_\delta^a)$. Applying $f$ to this equation, we obtain $N = \cl^{N} (f[M] \cup N_\delta)$, as desired.
  \end{enumerate}
\end{proof}

The next notion is studied explicitly in \cite[[V.E.1.2]{shelahaecbook2} and \cite[Definition 3.4]{bgkv-apal} (where it is called the minimal closure of $\nf$). It is a way to extend $\nf$ to take sets on the left and right hand side.

\begin{defin}\label{nfm-def}
  We write $\nfcl{M_0}{A}{B}{M_3}$ if $M_0 \lea M_3$, $A \cup B \subseteq |M_3|$, and there exists $M_3' \gea M_3$, $M_1 \lea M_3'$, and $M_2 \lea M_3'$ such that $A \subseteq |M_1|$, $B \subseteq |M_2|$, and $\nfs{M_0}{M_1}{M_2}{M_3'}$.
\end{defin}
\begin{lem}\label{nfcl-very-basic} \
  \begin{enumerate}
  \item $\nfs{M_0}{M_1}{M_2}{M_3}$ if and only if $\nfcl{M_0}{M_1}{M_2}{M_3}$ and $M_0 \lea M_\ell \lea M_3$ for $\ell = 1,2$.
  \item Invariance: if $\nfcl{M_0}{A}{B}{M_3}$ and $f: M_3 \cong M_3'$, then $\nfcl{f[M_0]}{f[A]}{f[B]}{M_3'}$.
  \item Monotonicity: if $\nfcl{M_0}{A}{B}{M_3}$ and $A_0 \subseteq A$, $B_0 \subseteq B$, and $M_3 \lea M_3'$, then $\nfcl{M_0}{A_0}{B_0}{M_3'}$.
  \end{enumerate}
\end{lem}
\begin{proof}
  Straight from the definitions.
\end{proof}

\begin{notation} \
  \begin{enumerate}
    \item When $N \in \K$, $M \lea N$, $B \subseteq |N|$, and $\ba \in \fct{<\infty}{|N|}$, we write $\nfcl{M}{\ba}{B}{N}$ for $\nfcl{M}{\operatorname{ran} (\ba)}{B}{N}$.
    \item For $p \in \gS^{<\infty} (B; N)$ and $M \lea N$, we say $p$ \emph{does not fork over $M$} if whenever $p = \gtp (\ba / B; N)$, we have that $\nfcl{M}{\ba}{B}{N}$. Note that this does not depend on the choices of representatives by Lemma \ref{nfcl-very-basic}.
  \end{enumerate}
\end{notation}

The following properties all appear either in \cite[Section 5.1]{bgkv-apal} or \cite[Sections 4,12]{indep-aec-apal}. We will use them without comments.

\begin{fact}\label{basic-indep-facts} \
  \begin{enumerate}
    \item Normality: If $\nfcl{M_0}{A}{B}{M_3}$, then $\nfcl{M_0}{A \cup |M_0|}{B \cup |M_0|}{M_3}$.
    \item Base monotonicity: if $\nfcl{M_0}{A}{B}{M_3}$ and $M_0 \lea M_0' \lea M_3$ is such that $|M_0'| \subseteq B$, then $\nfcl{M_0'}{A}{B}{M_3}$.
    \item Symmetry: If $\nfcl{M_0}{A}{B}{M_3}$, then $\nfcl{M_0}{B}{A}{M_3}$.
    \item Extension: Let $M \lea N$ and $B \subseteq C \subseteq |N|$ be given. If $p \in \gS^{<\infty} (B; N)$ does not fork over $M$, then there exists $N' \gea N$ and $q \in \gS^{<\infty} (C; N')$ extending $p$ and not forking over $M$.
    \item Uniqueness: Let $M \lea N$ and let $|M| \subseteq B \subseteq |N|$. If $p, q \in \gS^{<\infty} (B; N)$ do not fork over $M$ and $p \rest M = q \rest M$, then $p = q$.
    \item Transitivity: If $\nfcl{M_0}{A}{M}{N}$, $\nfcl{M}{A}{B}{N}$, and $M_0 \lea M$, then $\nfcl{M_0}{A}{B}{N}$.
  \end{enumerate}
\end{fact}

The following is a form of local character that $\nf$ may have:

\begin{defin}[V.C.3.7 in \cite{shelahaecbook2}]\label{based-def}
  We say that \emph{$\nf$ is $\chi$-based} if whenever $M \lea M^\ast$ and $A \subseteq |M^\ast|$ then there are $N_0$ and $N_1$ so that $\|N_1\| \le |A| + \chi$, $N_0 = M \cap N_1$, $A \subseteq |N_1|$, and $\nfs{N_0}{N_1}{M}{M^\ast}$.
\end{defin}

Interestingly, if $\nf$ is based then smoothness for small lengths implies smoothness for all lengths.

\begin{fact}[V.D.1.2 in \cite{shelahaecbook2}]\label{smoothness-upward-fact}
  If $\K$ is $(\le \LS (\K), \le \LS (\K)^+)$-smooth (recall Definition \ref{smooth-def}) and $\nf$ is $\LS (\K)$-based, then $\K$ is smooth, i.e.\ it is an AEC.
\end{fact}

A consequence of $\nf$ being based is that the class is tame. The argument is folklore and appears already in \cite[p.~15]{superior-aec}.

\begin{lem}\label{basic-indep-if-based}
  Assume that $\nf$ is $\LS (\K)$-based.
  
  \begin{enumerate}
    \item\label{basic-indep-set-lc} Set local character: if $p \in \gS^{<\infty} (M)$, then there are $M_0 \lea M$ such that $\|M_0\| \le |\ell (p) | + \LS (\K)$ and $p$ does not fork over $M_0$.
    \item\label{basic-indep-tame} $\K$ is $\LS (\K)$-tame.
  \end{enumerate}
\end{lem}
\begin{proof} \
  \begin{enumerate}
    \item Straight from the definitions.
    \item Combine (\ref{basic-indep-set-lc}) with the uniqueness property. 
  \end{enumerate}
\end{proof}

\section{Enumerated trees and generalized symmetry}\label{enum-trees-sec}

\begin{hypothesis}\label{tree-hyp-0}
  $(\K, \nf, \cl)$ satisfies $\Axfr$. Eventually, we will also assume Hypotheses \ref{tree-hyp} and \ref{good-indep-hyp}
\end{hypothesis}

Consider a minimal failure of smoothness: an increasing chain $\seq{M_i : i \le \delta}$ that is continuous below $\delta$ but so that $\bigcup_{i < \delta} M_i \not \lea M_\delta$. We would like to copy this chain into a tree indexed by $\fct{\le \delta}{\lambda}$. The branches of the tree should be as independent as possible. The main theorem of this section, Theorem \ref{tree-constr-thm}, shows that it can be done. We show in Theorem \ref{smoothness-unstable} that the resulting tree of failures witnesses unstability.

The main difficulty in the proof of Theorem \ref{tree-constr-thm} is that we cannot assume smoothness when we construct the tree, so we have difficulties at limits (because, to quote the referee, the tree is ``wider than it is high''). We work around this by studying trees enumerated in some order, giving a definition of a closed subset of such tree (Definition \ref{closed-def}) and proving a generalized symmetry theorem for these sets (Theorem \ref{generalized-sym}). Generalized symmetry says intuitively (as in \cite{sh87a,sh87b}) that whether a tree is independent does not depend on its enumeration, so closed sets will be as independent of each other as possible. Once generalized symmetry is proven, the construction of the desired tree can be carried out.

This section draws a lot of inspiration from \cite[V.C.4]{shelahaecbook2}, where Shelah defines a notion of stable construction which is supposed to accomplish similar goals than here. Shelah even states Theorem \ref{tree-constr-thm} as an exercise \cite[V.C.4.14]{shelahaecbook2}. However, we cannot solve it when smoothness fails. It seems that clause (vi) in \cite[Definition V.C.4.2]{shelahaecbook2} is too restrictive and precisely prevents us from copying a non-smooth chain into a tree.

We start by setting up the notation of this section for trees. The universe of the trees we will use is always an ordinal $\alpha$, and we think of $(\alpha, \le)$ as giving the order in which the tree is enumerated and $(\alpha, \tleq)$ as being the tree order.

\begin{defin}
  An \emph{enumerated tree} is a pair $(\alpha, \tleq)$, where $\alpha$ is an ordinal and $\tleq$ is a partial order on $\alpha$ such that for all $i, j < \alpha$:

  \begin{enumerate}
    \item $0 \tleq i$ (i.e.\ $0$ is the root of the tree).
    \item $i \tleq j$ implies $i \le j$ (i.e.\ if $j$ is above $i$ in the tree, then it is enumerated later).
    \item $(\{k < \alpha \mid k \tleq i\}, \tleq)$ is a well-ordering.
  \end{enumerate}
\end{defin}

\begin{defin}
  Let $(\alpha, \tleq)$ be an enumerated tree.
  
  \begin{enumerate}
    \item For $i < \alpha$, and $R \in \{\tlt, \tleq\}$, let $\pred_R (i) := \{k \le i \mid k R i\}$. When $R = \tlt$, we omit the subscript.
    \item A \emph{branch} of $(\alpha, \tleq)$ is a set $b \subseteq \alpha$ such that:
      \begin{enumerate}
        \item $\tleq$ linearly orders $b$.
        \item $i \in b$ implies $\pred (i) \subseteq b$.
      \end{enumerate}
    \item A branch $b \subseteq \alpha$ is \emph{bounded (in $(\alpha, \tleq)$)} if either it has a maximum or $b = \pred (i)$ for some $i < \alpha$. It is unbounded otherwise. We say that a set $u \subseteq \alpha$ is \emph{bounded} if any branch $b \subseteq u$ is bounded.
    \item We say that $(\alpha, \tleq)$ is \emph{continuous} when for any $i, j < \alpha$, if $\pred (i) = \pred (j)$ and $\pred (i)$ does not have a maximum, then $i = j$.
    \item When $(\alpha, \tleq)$ is continuous and $b \subseteq \alpha$ is a bounded branch, we let: 

      $$
      \topp (b) := \begin{cases}
        \max (b) & \text{if } b \text{ has a }\tleq\text{-maximum} \\
        \text{The unique }i < \alpha \text{ such that } b = \pred (i) & \text{otherwise.}
        \end{cases}
      $$
    \item When $u \subseteq \alpha$, let:

      $$
      B (u) := \{b \subseteq u \mid b \text{ is a branch and for any branch } b', b \subseteq b' \subseteq u\text{ implies } b' = b\}
      $$

      be the set of branches in $u$ that are maximal in $u$.
    \item When $(\alpha, \tleq)$ is continuous and $u \subseteq \alpha$ is a bounded set, we let $\topp (u) := \sup_{b \in B (u)} \topp (b)$ (this will only be used when $B (u)$ is finite, so in that case the supremum is actually a maximum).
  \end{enumerate}
\end{defin}

\begin{lem}\label{branch-injection}
  If $u \subseteq v$ and $b \in B (u)$, then there is a branch $b' \in B (v)$ such that $b \subseteq b'$. Consequently, $|B (u)| \le |B (v)|$.
\end{lem}
\begin{proof}
  Straightforward from the definition of $B (u)$. The last sentence is because the map $b \mapsto b'$ (for some choice of $b'$) is an injection from $B (u)$ to $B (v)$.
\end{proof}

We now define a tree of structures coming from the class $\K$. Note that continuity of chains of models is only required when the chain is smooth (see (\ref{cont-tree-def-4}) below).

\begin{defin}\label{cont-tree-def}
  A \emph{continuous enumerated tree of models} is a tuple $(\seq{M_i :i < \alpha}, N, \alpha, \tleq)$ satisfying:

  \begin{enumerate}
  \item $(\alpha, \tleq)$ is a continuous enumerated tree.
    \item $N \in \K$.
    \item\label{tree-def-2} For all $i < \alpha$, $M_i \lea N$.
    \item For all $i, j < \alpha$, $i \tleq j$ implies $M_i \lea M_j$.
    \item\label{cont-tree-def-4} For all $i < \alpha$, if $\pred (i)$ has no maximum \emph{and} $\bigcup_{j \tlt i} M_j \lea N$, then $M_i = \bigcup_{j \tlt i} M_j$.
  \end{enumerate}
\end{defin}

\begin{remark}
  By coherence, for all $i < \alpha$, $\bigcup_{j \tlt i} M_j \lea N$ if and only if $\bigcup_{j \tlt i} M_j \lea M_i$
\end{remark}
\begin{remark}
  In Definition \ref{cont-tree-def}, $N$ is just an ambient model. Eventually, we will want to also ensure that it satisfies a minimality condition (see the conclusion of Theorem \ref{tree-constr-lem}).
\end{remark}

From now on until Lemma \ref{kb-constr}, we assume:

\begin{hypothesis}\label{tree-hyp}
   $\mathcal{T} := (\seq{M_i :i < \alpha}, N, \alpha, \tleq)$ is a continuous enumerated tree of models.
\end{hypothesis}

The following is a key definition. Intuitively, a set is closed if it is closed under initial segments \emph{and} all its branches smoothly embed inside $N$.

\begin{defin}\label{closed-def}
  $u \subseteq \alpha$ is \emph{closed} if:

  \begin{enumerate}
    \item\label{closed-1} $i \in u$ implies $\pred (i) \subseteq u$.
    \item\label{closed-2} $b \in B (u) \backslash \{\emptyset\}$ implies $\bigcup_{i \in b} M_i \lea N$.
  \end{enumerate}
\end{defin}

\begin{lem}\label{closed-closure} \
  \begin{enumerate}
    \item An arbitrary intersection of closed sets is closed.
    \item A finite union of closed sets is closed.
  \end{enumerate}
\end{lem}
\begin{proof} \
  \begin{enumerate}
    \item Let $\seq{u_i : i < \gamma}$ be closed, $\gamma > 0$. Let $u := \bigcap_{i < \gamma} u_i$. We show that $u$ is closed. It is easy to check that $u$ satisfies (\ref{closed-1}) from the definition of a closed set. We check (\ref{closed-2}). Let $b \in B (u) \backslash \{\emptyset\}$. We want to see that $\bigcup_{j \in b} M_j \lea N$. By Lemma \ref{branch-injection}, for each $i < \gamma$ there exists $b_i \in B (u_i)$ such that $b \subseteq b_i$. Since $u_i$ is closed, we have that $\bigcup_{j \in b_i} M_j \lea N$. If there exists $j < \gamma$ such that $b = b_j$, we are done so assume that this is not the case. This implies that $b$ is bounded. Let $k := \topp (b)$. We know that $b \subsetneq b_j$ for all $j < \gamma$, so by downward closure we must have that $k \in b_j$ for all $j < \gamma$. But then this means that $k \in u$, so $k \in b$, a contradiction. 
    \item Let $u, v$ be closed. We show that $u \cup v$ is closed. As before, (\ref{closed-1}) is straightforward to see. As for (\ref{closed-2}), let $b \in B (u \cup v)$. It is straightforward to see that either $b \in B (u)$ or $b \in B (v)$. In either case we get that $\bigcup_{i \in b} M_i \lea N$, as desired.
  \end{enumerate}
\end{proof}
\begin{remark}
  Lemma \ref{closed-closure} almost tells us that closed sets induce a topology on $\alpha$. While it is easy to check that the empty set is closed, $\alpha$ itself may not be closed (think of a chain $\seq{M_i : i \le \delta}$ where $\bigcup_{i < \delta} M_i \not \lea M_\delta$. The tree could consist of $\seq{M_i : i < \delta}$ and $N = M_\delta$). However $\alpha$ \emph{will} be closed when all the maximal branches of the tree have a maximum (e.g.\ if $(\alpha, \tleq)$ looks like $\fct{\le \delta}{\lambda}$ for some cardinal $\lambda \ge 2$ and limit ordinal $\delta$).
\end{remark}

The next definition describes the model $M^u$ generated by a set $u \subseteq \alpha$. Typically, $u$ will be closed and in case the tree is sufficiently independent (see Definition \ref{indep-def}), $M^u$ will be in $\K$.

\begin{defin}
  For $u \subseteq \alpha$, $M^u := \cl^{N} (|M_0| \cup \bigcup_{i \in u} |M_i|)$.
\end{defin}
\begin{lem}\label{basic-mu-calculus}
Let $u, v \subseteq \alpha$. $M^{u \cup v} = \cl^{N} (M^u \cup M^v)$.
\end{lem}
\begin{proof}
By Fact \ref{cont-of-cl}.
\end{proof}

\begin{lem}\label{cb-union}
  If $b$ is a closed and bounded branch, then $M^b = M_i$, where $i := \topp (b)$.
\end{lem}
\begin{proof}
  If $i$ is a maximum of $b$ or $b$ is empty, this is clear. If not, we know since $b$ is closed that $\bigcup_{j \tlt i} M_j = \bigcup_{j \in b} M_j \lea N$. By (\ref{cont-tree-def-4}) in Definition \ref{cont-tree-def}, $M_i = \bigcup_{j \tlt i} M_j$. Note that $\bigcup_{j \tlt i} M_j = \cl^N (\bigcup_{j \tlt i} M_j)$ and by Fact \ref{cont-of-cl}, this is equal to $M^b$. So $\bigcup_{j \tlt i} M_j = M^b$, as desired.
\end{proof}

The next definition describes when two (typically closed) sets $u$ and $v$ are ``as independent as possible'', i.e.\ the model generated by $u$ is independent of the one generated by $v$ over the model generated by $u \cap v$. There are two variations depending on whether the ambient model is $N$ or the model generated by $u \cup v$.

Generalized symmetry (Theorem \ref{generalized-sym}) will say that under appropriate conditions, if the tree is independent then any closed sets $u$ and $v$ are as independent as possible.

\begin{defin}
  Let $u, v \subseteq \alpha$.
  
  \begin{enumerate}
  \item We write $uv$ for $u \cup v$.
  \item We write $u \nf v$ if $\nfs{M^{u \cap v}}{M^u}{M^v}{M^{u v}}$.
  \item We write $u \nf^N v$ if $\nfs{M^{u \cap v}}{M^u}{M^v}{N}$.
  \end{enumerate}
\end{defin}

Note that to make the notation lighter we omit the base and write $u \nf v$ instead of e.g.\ $u \nf_{u \cap v} v$.

The following will be used without comment.

\begin{lem}
  $u \nf^N v$ if and only if [$u \nf v$ and $M^{uv} \lea N$].
\end{lem}
\begin{proof}
  If $u \nf^N v$, then by Fact \ref{basic-axfr-props}.(\ref{basic-props-1}), $M^{uv} \lea N$ and $u \nf v$. The converse is by the monotonicity 2 property of $\nf$ in Definition \ref{axfr-def}.
\end{proof}

If $u \subseteq v$, there is an easy way to determine whether $u \nf v$.

\begin{lem}\label{trivial-subset-result}
  If $u \subseteq v$, $M^u \lea N$, and $M^v \lea N$, then $u \nf^N v$.
\end{lem}
\begin{proof}
  Straight from the definition.
\end{proof}

We now translate the properties of Section \ref{axfr-sec} into properties of the relations $u \nf v$ and $u \nf^N v$.

\begin{lem}\label{basic-uv-props}
  Let $u,v, w \subseteq \alpha$ be closed.

  \begin{enumerate}
    \item Symmetry: If $u \nf v$, then $v \nf u$. If $u \nf^N v$, then $v \nf^N u$.
    \item\label{base-enlarg} Base enlargement: If $u \nf^N v$, $u \cap v \subseteq w \subseteq v$, and $M^w \lea M^v$, then $uw \nf^N v$.
    \item\label{trans} Transitivity: If $u \nf^N v$, $uv \nf^N w$, $uv \cap w = u$, and $M^{v \cap w} \lea M^w$, then $v \nf^N w$.
  \end{enumerate}
\end{lem}
\begin{proof} \
  \begin{enumerate}
    \item Straightforward from the symmetry axiom.
    \item Directly from the base enlargement axiom (note that $uw \cap v = w$), see Definition \ref{axfr-def}.
    \item Let $M_0 := M^{u \cap v}$, $M_1 := M^v$, $M_2 := M^u$, $M_3 := M^{uv}$, $M_4 := M^w$, $M_5 := N$. We know that $u \nf^N v$, so $M^{uv} \lea N$ and $u \nf v$, hence $\nfs{M_0}{M_1}{M_2}{M_3}$ holds. We know that $uv \nf^N w$ (so in particular $M^{uvw} \lea N$) and $uv \cap w = u$, i.e.\ $M^{uv \cap w} = M^u = M_2$, so $\nfs{M_2}{M_3}{M_4}{M_5}$ holds. Applying Fact \ref{basic-axfr-props}.(\ref{basic-props-2}), we obtain $\nfs{M_0}{M_1}{M_4}{M_5}$, i.e.\ $\nfs{M^{u \cap v}}{M^v}{M^w}{N}$. Now since $uv \cap w = u$, we must have that $u \subseteq w$ and $v \cap w \subseteq u$. Therefore $u \cap v \subseteq w \cap v$. By coherence, $M^{u \cap v} \lea M^{v \cap w} \lea M^w$. By base enlargement, $\nfs{M^{v \cap w}}{M^v}{M^w}{N}$, i.e.\ $v \nf^N w$.
  \end{enumerate}
\end{proof}

A key part of the proof of generalized symmetry is a concatenation property telling us when $uv \nf^N w$ if we know something about $u$ and $v$ separately. We start with the following result:

\begin{lem}\label{pre-concat-lem}
  Let $u, v, w \subseteq \alpha$ be closed. If:
  
  \begin{enumerate}
  \item\label{hyp-1} $u(v \cap w) \nf ^N w$.
  \item\label{hyp-2} $v \nf^N uw$.
  \item\label{hyp-3} $M^{u(v \cap w)} \lea M^{uw}$.
  \item\label{hyp-4} $M^{uv \cap w} \lea M^{uv}$.
  \end{enumerate}

  Then $uv \nf^N w$.
\end{lem}
\begin{proof}
  We apply base enlargement with $u, v, w$ in \ref{basic-uv-props}.(\ref{base-enlarg}) standing for $v, uw, u(v \cap w)$ here. The hypotheses hold by (\ref{hyp-2}) and (\ref{hyp-3}). We obtain $uv \nf^N uw$. We want to apply transitivity, where $u, v, w$ in \ref{basic-uv-props}.(\ref{trans}) stand for $u(v \cap w)$, $w$, $uv$ here. The conditions there are:
      \begin{itemize}
        \item $u \nf^N v$, which translates to $u(v \cap w) \nf^N w$ here (holds by (\ref{hyp-1})).
        \item $uv \nf^N w$, which translates to $uw \nf^N uv$ here (holds by the paragraph above and symmetry).
        \item $uv \cap w = u$, which translates to $u(v \cap w)w \cap uv = u (v \cap w)$, i.e.\ $uw \cap uv = u (v \cap w)$, which is true.
        \item $M^{v \cap w} \lea M^w$, which translates to $M^{w \cap uv} \lea M^{uv}$, which is true by (\ref{hyp-4}).
      \end{itemize}

      Therefore the conclusion of transitivity holds. In our case, this means that $w \nf^N uv$. By symmetry, $uv \nf^N w$, as desired.
\end{proof}

\begin{lem}\label{pre-concat-2-lem}
  Let $u, v, w \subseteq \alpha$ be closed. If:
  \begin{enumerate}
  \item\label{pre-concat-2-1} $u \nf^N w$.
  \item\label{pre-concat-2-2} $M^{uv \cap w} \lea N$.
  \item\label{pre-concat-2-3} $M^{u (v \cap w)} \lea M^{u(v \cap w)}$.
  \end{enumerate}
  
  Then $u (v \cap w) \nf^N w$.  
\end{lem}
\begin{proof}
  We use Lemma \ref{pre-concat-lem} with $u, v, w$ there standing for $w \cap v$, $u$, $w$ here. Let us check the hypotheses:
  
  \begin{itemize}
    \item (\ref{hyp-1}) there translates to $(w \cap v) (u \cap w) \nf^N w$ here. So it is enough to see that $M^w \lea N$ and $M^{uv \cap w} \lea N$. This holds by (\ref{pre-concat-2-1}) and (\ref{pre-concat-2-2}).
    \item (\ref{hyp-2}) there translates to $u \nf^N w$ here, which is (\ref{pre-concat-2-1}).
    \item (\ref{hyp-3}) there translates to $M^{uv \cap w} \lea M^w$ here. This holds by (\ref{pre-concat-2-1}), (\ref{pre-concat-2-2}), and coherence.
    \item (\ref{hyp-4}) there translates to $M^{uv \cap w} \lea M^{u (v \cap w)}$ here. This holds by (\ref{pre-concat-2-3}).
  \end{itemize}

  The hypotheses hold, so we obtain that $u (v \cap w) \nf^N w$, as needed.
\end{proof}

Finally, we obtain a usable concatenation property.

\begin{lem}[Concatenation]\label{concat-lem}
  Let $u, v, w \subseteq \alpha$ be closed. If:
  
  \begin{enumerate}
  \item\label{concat-hyp-1} $u \nf ^N w$.
  \item\label{concat-hyp-2} $v \nf^N uw$.
  \item\label{concat-hyp-3} $M^{uv \cap w} \lea N$.
  \item\label{concat-hyp-4} $M^{u (v \cap w)} \lea N$.
  \item\label{concat-hyp-5} $M^{uv} \lea N$.
  \end{enumerate}

  Then $uv \nf^N w$.
\end{lem}
\begin{proof}
  We use Lemma \ref{pre-concat-lem}. Let us check the hypotheses:
  
  \begin{itemize}
    \item (\ref{hyp-1}) says $u(v \cap w) \nf^N w$. This holds by Lemma \ref{pre-concat-2-lem}. Note that (\ref{pre-concat-2-1}) there holds by (\ref{concat-hyp-1}), (\ref{pre-concat-2-2}) there holds by (\ref{concat-hyp-3}), and (\ref{pre-concat-2-3}) there holds by (\ref{concat-hyp-3}), (\ref{concat-hyp-4}), and coherence.
    \item (\ref{hyp-2}) there is (\ref{concat-hyp-2}) here.
    \item (\ref{hyp-3}) there is given by (\ref{concat-hyp-1}), (\ref{concat-hyp-4}), and coherence.
    \item (\ref{hyp-4}) there is given by (\ref{concat-hyp-3}), (\ref{concat-hyp-5}), and coherence.
  \end{itemize}

  The hypotheses hold, so we obtain that $uv \nf^N w$, as needed.
\end{proof}

Another key ingredient of the proof of generalized symmetry is a continuity property that tells us how to deal with increasing chains $\seq{u_i : i < \delta}$ of closed sets. At that point, the following hypothesis will appear in some of the statements (we do \emph{not} assume it globally).

\begin{defin}\label{cl-algebraic}
  We say that \emph{$\cl$ is algebraic} if for any $M, N \in \K$ with $M \subseteq N$ and any $A \subseteq |M|$, $\cl^M (A) = \cl^N (A)$.
\end{defin}

Recall that we are working under Hypothesis \ref{tree-hyp-0}, so $\cl$ is in particular a fixed operator satisfying Monotonicity 2 (Definition \ref{axfr-def}.(\ref{axfr-def-mon2})). The difference here is that we assume that closure is the same whenever $M \subseteq N$ (not only under the stronger condition $M \lea N$).

Note that if $\cl^N (A)$ is the closure of $A$ under the functions of $N$, then $\cl$ is algebraic. This will be the closure operator when we study universal classes, so we do not lose much by assuming it here. In fact, we could have assumed from the beginning that $\cl^N (A) $ was the closure of $A$ under the functions of $N$. For the purpose of proving the main result of this paper, we would not lose anything.

\begin{lem}\label{basic-mu-calculus-cont}
  Assume that $\cl$ is algebraic. Let $\delta$ be a limit ordinal and let $\seq{u_i : i \le \delta}$ be an increasing continuous chain of closed sets. If for all $i < \delta$, $M^{u_i}$ is a $\tau (\K)$-structure, then $M^{u_\delta} = \bigcup_{i < \delta} M^{u_i}$.
\end{lem}
\begin{proof}
  Let $M_\delta := \bigcup_{i < \delta} M^{u_i}$.

  First observe that $M_\delta \subseteq N$, because for all $i < \delta$, $M^{u_i} \subseteq N$ (as we are assuming it is a $\tau (\K)$-structure and by definition it must inherit the function symbols from $N$). Therefore because $\cl$ is algebraic, $\cl^{N} (M_\delta) = \cl^{M_\delta} (M_\delta) = M_\delta$. But $\cl^N (M_\delta) = \cl^N \left(\bigcup_{i < \delta} \cl^{N} (M_0 \cup \bigcup_{j \in u_i} M_j)\right)$. By Fact \ref{cont-of-cl}, this is just $\cl^N \left(\bigcup_{i < \delta} (M_0 \cup \bigcup_{j \in u_i} M_i)\right) = \cl^N (M_0 \cup \bigcup_{i \in u_\delta} M_i) = M^{u_\delta}$. Combining the chains of equalities, we have the result.
\end{proof}

\begin{lem}[Continuity]\label{cont-lemma}
  Assume that $\cl$ is algebraic.
  
  Let $\delta$ be a limit ordinal and let $\seq{u_i : i \le \delta}$, $\seq{v_i : i \le \delta}$ be increasing continuous chains of closed sets. If for all $i, j < \delta$: 

  \begin{enumerate}
    \item\label{cont-hyp-1} $u_i \nf v_j$.
    \item\label{cont-hyp-2} $u_\delta \cap v_i \nf v_j$.
    \item\label{cont-hyp-3} $v_\delta \cap u_i \nf u_j$.
  \end{enumerate}

  Then $u_\delta \nf v_\delta$.
\end{lem}
\begin{proof}
  \underline{Claim 1}: $M^{u_\delta \cap v_\delta} \lea M^{v_\delta}$. 

  \underline{Proof of Claim 1}: We use Fact \ref{basic-axfr-props}.(\ref{basic-props-3}) where $M_i, N_i$ there stand for $M^{u_\delta \cap v_i}$, $M^{v_i}$ here. Why is $\seq{M^{u_\delta \cap v_i} : i \le \delta}$ $\subseteq$-increasing and continuous? Note that $M^{u_\delta \cap v_i}$ is a member of $\K$ for each $i < \delta$ (by (\ref{cont-hyp-2})), and the chain is increasing by definition of $M^{u_\delta \cap v_i}$. The continuity is because $\seq{v_i : i \le \delta}$ is itself continuous (use Lemma \ref{basic-mu-calculus-cont}). Similarly, $\seq{M^{v_i} : i \le \delta}$ is $\subseteq$-increasing continuous. Also, (\ref{cont-hyp-2}) ensures that the independence hypothesis of Fact \ref{basic-axfr-props}.(\ref{basic-props-3}) is satisfied. Therefore we have in particular that $M_\delta \lea N_\delta$ there. That is, $M^{u_\delta \cap v_\delta} \lea M^{v_\delta}$. $\dagger_{\text{Claim}}$.

  \underline{Claim 2}: For all $j < \delta$, $u_j \nf v_\delta$.
  
  \underline{Proof of Claim 2}: Fix $j < \delta$. We will show that $v_\delta \nf u_j$. For this, we use Fact \ref{basic-axfr-props}.(\ref{basic-props-4}) where $M_i$, $M_{\delta + 1}$, $N_i^a$, $N_i^b$ there stand for $M^{u_j \cap v_{j + i}}$, $M^{u_j}$, $M^{v_{j + i}}$, $M^{u_j v_{j + i}}$ here (so we see $M_{\delta + 1}$ as really the ``fixed'' part and the $N_i^a$'s as the ``growing'' part). All the hypotheses of Fact \ref{basic-axfr-props}.(\ref{basic-props-4}) are satisfied. In detail, we have to check that there $M_\delta \lea M_{\delta + 1}$, which here translates to $M^{u_j \cap v_{\delta}} \lea M^{u_j}$, but this holds by (\ref{cont-hyp-3}). Also, $N_i^b = \cl^{N_i^b} (M_{\delta + 1} \cup N_i^a)$ there translates to $M^{u_j v_{j + i}} = \cl^{M_{u_j v_{j + i}}} (M_{u_j} \cup M^{v_{j + i}})$. This holds because $M^{u_j v_{j + i}} \subseteq N$ (by (\ref{cont-hyp-1}), $M^{u_j v_{j + i}} \in \K$, and hence by definition it must be a substructure of $N$), and hence because $\cl$ is algebraic, $\cl^N (A) = \cl^{M^{u_j v_{j + i}}} (A)$ for any set $A$. The other conditions are checked similarly. Applying Fact \ref{basic-axfr-props}.(\ref{basic-props-4}), we obtain that $v_\delta \nf u_j$, and hence by symmetry $u_j \nf v_\delta$ as desired. $\dagger_{\text{Claim 2}}$.

  To prove that $u_\delta \nf v_\delta$, we use Fact \ref{basic-axfr-props}.(\ref{basic-props-4}) again where $M_i, M_{\delta  + 1}$, $N_i^a$, $N_i^b$ there stand for $M^{u_i \cap v_i}$, $M^{v_\delta}$, $M^{u_i}$, $M^{u_i v_\delta}$ here. We need to know there that $M_\delta \lea M_{\delta + 1}$, i.e.\ $M^{u_\delta \cap v_\delta} \lea M^{v_\delta}$, but this is given by Claim 1. Further by Claim 2, $u_i \nf v_\delta$ for every $i < \delta$, so the hypotheses of Fact \ref{basic-axfr-props}.(\ref{basic-props-4}) hold.
\end{proof}

With the forking calculus out of the way, we are ready to start proving generalized symmetry. First, we state what it means for a tree to be independent. The intuition is that for any $i \tleq j$, $M_j$ is independent over $M_i$ of as much as possible that comes before $j$ in the enumeration of the tree. We use a slightly different notation than in e.g.\ \cite{sh87a, sh87b} but the notion described is the same.

\begin{defin}\label{indep-def}
  $\mathcal{T}$ is \emph{independent} if for any $i \tleq j < \alpha$: 

  $$
  \nfcl{M_i}{M_j}{\bigcup_{k \in A_{i, j}} M_k}{N}
  $$

  where $\nfm$ is from Definition \ref{nfm-def} and:

  $$
  A_{i, j} := \{k < j \mid \pred_{\tleq} (k) \cap \pred_{\tleq} (j) \subseteq \pred_{\tleq} (i)\}
  $$
\end{defin}

From now on until Lemma \ref{kb-constr}, we assume:

\begin{hypothesis}\label{good-indep-hyp}
  $\mathcal{T}$ (from Hypothesis \ref{tree-hyp}) is independent.
\end{hypothesis}

Our aim is to prove Theorem \ref{generalized-sym} which gives conditions under which $u \nf v$ for any closed sets $u$ and $v$. We prove increasingly stronger approximations to this result, each time using the previously proven approximations. First, we prove it when $u$ and $v$ are closed bounded branches.

\begin{lem}\label{two-branches}
  If $a$ and $b$ are closed bounded branches, then $a \nf^N b$.
\end{lem}
\begin{proof}
  Let $i := \topp (a)$, $j := \topp (b)$. By Lemma \ref{cb-union}, $M^a = M_i$, $M^b = M_j$. By Definition \ref{cont-tree-def}.(\ref{tree-def-2}), $M^a \lea N$ and $M^b \lea N$. Note that $a \cap b$ is also a closed bounded branch so $M^{a \cap b} \lea N$ also. By coherence, $M^{a \cap b} \lea M^x$ for $x \in \{a, b\}$. By symmetry, we can assume without loss of generality that $j \le i$. Furthermore, if $i = j$ then Lemma \ref{trivial-subset-result} gives the result, so assume $j < i$. Let $k := \topp (a \cap b)$. By Lemma \ref{cb-union} again, $M^{a \cap b} = M_k$. Now by Definition \ref{indep-def}, we must have that $\nfcl{M_k}{M_i}{M_j}{N}$. By what we have argued, we must actually have $\nfs{M_k}{M_i}{M_j}{N}$, i.e.\ $a \nf^N b$, as needed.
\end{proof}

Next, we prove it when $u$ is a closed and bounded branch and $v$ is a bounded finite union of closed branches that comes before $u$ in the enumeration of the tree (see Condition (\ref{top-cond}) below).

\begin{lem}\label{one-branch-one-ordered}
  If:
  
  \begin{enumerate}
    \item $a$ is a closed and bounded branch.
    \item $v$ is a closed and bounded set with $B (v)$ finite.
    \item\label{top-cond} $\topp (a) \ge \topp (v)$.
  \end{enumerate}

  Then $a \nf^N v$.
\end{lem}
\begin{proof}
  Let $n := |B (v)|$. We work by induction on $n$. If $n = 1$, the result holds by Lemma \ref{two-branches}. Otherwise, say $B(v) = \{b_0, \ldots, b_{n - 1}\}$, where without loss of generality $\topp (b_0) < \topp (b_1) < \ldots < \topp (b_{n - 1})$. By the induction hypothesis, $b_{n - 1} \nf^N b_0 \ldots b_{n - 2}$. In particular, $M^v = M^{b_0 \ldots b_{n - 1}} \lea N$. Now using Definition \ref{indep-def} (or Lemma \ref{trivial-subset-result} if $\topp (a) = \topp (v)$, so $a \subseteq v$), it is easy to check that $\nfcl{M^{a \cap v}}{M^a}{M^v}{N}$, so the result follows.
\end{proof}

Next, we can show that $M^u \lea N$ when $u$ is a bounded finite union of closed branches.

\begin{lem}\label{lt-finite}
  If $u$ and $v$ are bounded closed sets with $B (u)$ and $B (v)$ both finite, then:
  \begin{enumerate}
    \item $M^u \lea N$.
    \item $u \subseteq v$ implies $M^u \lea M^v$.
  \end{enumerate}
\end{lem}
\begin{proof}
  The second part follows from the first and coherence. For the first part, let $n := |B (u)|$ and write $B(u) = \{b_0, \ldots, b_{n - 1}\}$ with $\topp (b_0) < \ldots < \topp (b_{n - 1})$. If $n = 1$, the result follows from Lemma \ref{two-branches} (where $a, b$ there stand for $u, u$ here) so assume that $n \ge 2$. Apply Lemma \ref{one-branch-one-ordered} where $a, v$ there stand for $b_{n - 1}$, $b_0 \ldots b_{n - 2}$ here.
\end{proof}

We now use the previous result together with concatenation to show that $u \nf^N v$ when $u$ and $v$ are bounded finite union of closed branches.

\begin{lem}\label{concat-step-3}
  If $u$ and $v$ are closed bounded sets with $B (u)$ and $B (v)$ both finite, then $u \nf^N v$.
\end{lem}
\begin{proof}
  Work by induction on $|B (u)| + |B (v)|$. By symmetry, without loss of generality $\topp (u) \ge \topp (v)$. Let $n := |B (u)|$. Write $B (u) = \{a_0 , \ldots, a_{n - 1}\}$ with $\topp (a_0) < \ldots < \topp (a_{n - 1})$.  If $n = 1$, the result is given by Lemma \ref{one-branch-one-ordered}, so assume now that $n \ge 2$. We use concatenation (Lemma \ref{concat-lem}) with $u, v, w$ there standing for $a_0 \ldots a_{n - 2}$, $a_{n - 1}$, $v$ here. Let us check the hypotheses:

  \begin{itemize}
    \item (\ref{concat-hyp-1}) there translates to $a_0 \ldots a_{n - 2} \nf^N v$ here. This holds by the induction hypothesis.
    \item (\ref{concat-hyp-2}) there translates to $a_{n - 1} \nf^N a_0 \ldots a_{n - 2} v$ here. This holds by Lemma \ref{one-branch-one-ordered}.
    \item (\ref{concat-hyp-3})-(\ref{concat-hyp-5}) there hold by Lemma \ref{lt-finite}.
  \end{itemize}

  The hypotheses hold, so we obtain $a_0 \ldots a_{n - 1} \nf^N v$, as desired.
\end{proof}

Next, we can use the continuity property to prove generalized symmetry for all closed bounded sets.

\begin{lem}\label{concat-step-4}
  Assume that $\cl$ is algebraic. If $u$ and $v$ are closed bounded sets, then $u \nf v$.
\end{lem}
\begin{proof}
  Let $\lambda := |B (u \cup v)|$. We work by induction on $\lambda$. If $\lambda < \aleph_0$, then this is taken care of by Lemma \ref{concat-step-3}. Otherwise, say $B (u) = \seq{a_i : i < \lambda}$ and $B (v) = \seq{b_i : i < \lambda}$ (we allow repetition in the enumerations). For $i \le \lambda$, let $u_i := \bigcup_{j < i} b_j$ and $v_i := \bigcup_{j < i} b_j$. It is easy to check that $\seq{u_i : i \le \lambda}$, $\seq{v_i : i \le \lambda}$ are increasing continuous resolutions of $u$ and $v$ respectively. Moreover, each member of the chain is a closed bounded set. We apply Lemma \ref{cont-lemma} (where $\delta$ there stands for $\lambda$ here). Its hypotheses hold by the induction hypothesis. We obtain that $u_\lambda \nf v_\lambda$, as desired.
\end{proof}

When $u$ or $v$ is not bounded, we will make an additional hypothesis which says that branches do not have too many non-smooth points. In the case we are interested in (see Theorem \ref{tree-constr-thm}), each branch will have at most one nonsmooth point, so this hypothesis is reasonable. Note again that we \emph{do not} assume this globally, only in some statements. 

\begin{defin}
  $\mathcal{T}$ is \emph{resolvable} if for any branch $b \subseteq \alpha$, $\{i \in b \mid \bigcup_{j \tlt i} M_j \not \lea N\}$ is finite.
\end{defin}

\begin{defin}
  For $u \subseteq \alpha$, let $B' (u) := \{b \in B (u) \mid b \text{ is unbounded}\}$.
\end{defin}

Assuming that $\mathcal{T}$ is resolvable, we show that every closed set has a resolution with fewer unbounded branches than the original set. This will allow us to do a proof by induction on $|B' (u)|$.

\begin{lem}\label{resolution-lem}
  Assume that $\mathcal{T}$ is resolvable. 

  \begin{enumerate}
    \item Let $b$ be a closed branch. Then there is a limit ordinal $\delta$ and an increasing continuous sequence of closed bounded branches $\seq{b_i : i \le \delta}$ such that $b = b_\delta$.
    \item Let $u$ be a closed unbounded set. Then there is a limit ordinal $\delta$ and an increasing continuous sequence $\seq{u_i : i \le \delta}$ of closed sets such that $u_\delta = u$ and for all $i < \delta$, $|B' (u_i)| < |B' (u)|$.
  \end{enumerate}
\end{lem}
\begin{proof} \
  \begin{enumerate}
    \item If $b$ is bounded, we can take $b = b_i$ for all $i \le \delta$, so assume that $b$ is unbounded. Since $\mathcal{T}$ is resolvable, we know that there exists $i \in b$ such that for all $i' \ge i$, $\bigcup_{j \tlt i'} M_j \lea N$. In other words, $\pred (i')$ is closed. So let $\delta := \otp (b)$ and write $b \backslash i = \seq{i_j : j < \delta}$. For $j < \delta$, let $b_j := \pred (i_j)$.
    \item Say $B' (u) = \{b_i : i < \lambda\}$. Let $v := u \backslash \bigcup_{i < \lambda} b_i$. Note that $v$ is closed and bounded. If $\lambda$ is infinite, we can let $\delta := \lambda$ and for $i \le \delta$, $u_i := v \cup \bigcup_{j < i} b_j$. So assume that $\lambda$ is finite. By the first part, for each $i < \lambda$ there exists a limit ordinal $\delta_i$ and a resolution $\seq{b_i^j : j < \delta_i}$ of $b_i$ into closed bounded branches. Let $\delta := \sum_{i < \lambda} \delta_i$. Now for $j < \delta$, there are unique $i < \lambda$ and $k < \delta_i$ such that $j = \sum_{i_0 < i} \delta_{i_0} + k$. Set $u_j := \bigcup_{i_0 < i} b_{i_0} \cup b_i^{k}$. It is straightforward to check that this works.
  \end{enumerate}
\end{proof}

\begin{thm}[Generalized symmetry]\label{generalized-sym}
  Assume that $\mathcal{T}$ is resolvable and $\cl$ is algebraic. If $u$ and $v$ are closed sets, then $u \nf v$.
\end{thm}
\begin{proof}
  Work by induction on $\lambda := |B' (u)| + |B' (v)|$. If $\lambda = 0$, this is given by Lemma \ref{concat-step-4}. If $\lambda$ is infinite, we can use an argument analogous to the proof of Lemma \ref{concat-step-4}, so assume that $\lambda$ is finite and non-zero.

  By Lemma \ref{resolution-lem}, we can find limit ordinals $\delta_1, \delta_2$ and $\seq{u_i : i \le \delta_1}$, $\seq{v_i : i \le \delta_2}$ that are increasing continuous resolutions of $u$ and $v$ respectively so that each member in the chain is closed, and for all $i < \delta_1$, $|B' (u_i)| < |B' (u)|$, and similarly for $v$.
  
  By symmetry, without loss of generality, $\delta_1 \le \delta_2$. We first use Lemma \ref{cont-lemma} with $\delta$ there standing for $\delta_1$ here. The hypotheses hold by the induction hypothesis. So we obtain $u \nf v_{\delta_1}$. If $\delta_1 = \delta_2$, we are done. Otherwise by the induction hypothesis (using that $\lambda$ is finite) we have that $u \nf v_i$ for all $i < \delta_2$. So we use Lemma \ref{cont-lemma} a second time with $\delta$, $u_i$, $v_i$ there standing for $\delta_2$, $u$, $v_i$ here. We obtain that $u \nf v_{\delta_2}$, as desired.
\end{proof}

For the remainder of this section, we focus on building independent trees. We ``start from scratch'' and drop Hypotheses \ref{tree-hyp} and \ref{good-indep-hyp}. It will be convenient to have the tree enumerated in a particular order:

\begin{defin}\label{kb-def}
  An enumerated tree $(\alpha, \tleq)$ is \emph{in pre-order} if for any $i < \alpha$ and any $b \in B (i)$, either $b = \pred (i)$ or $b \in B (\alpha)$.
\end{defin}

The idea is that (Lemma \ref{kb-lem}) if the tree is in pre-order, then the set $A_{i, j}$ from Definition \ref{indep-def} is closed, so we can use generalized symmetry (Theorem \ref{generalized-sym}) on it. Before proving this, we show that the tree we care about has an enumeration in pre-order. For this, we simply keep building the same branch until it becomes maximal, then start a different branch.

\begin{lem}\label{kb-constr}
  Let $\delta$ be a limit ordinal and let $\lambda$ be a cardinal with $\lambda \ge 2$. Then there exists an enumeration $\seq{\eta_i : i < \alpha}$ of $\fct{\le \delta}{\lambda}$ such that defining $i \tleq j$ if and only if $\eta_i$ is an initial segment of $\eta_j$, we have that $(\alpha, \tleq)$ is a continuous enumerated tree which is in pre-order.
\end{lem}
\begin{proof}
  Let $\seq{\nu_j : j < \beta}$ be an enumeration (without repetitions) of $\fct{\le \delta}{\lambda}$ such that if $\nu_j$ is an initial segment of $\nu_{j'}$, then $j \le j'$. We define $\alpha$ and $\seq{\eta_i : i < \alpha}$ by induction on $i$ such that:

  \begin{enumerate}
    \item $(i, \tleq)$ is a continuous enumerated tree.
    \item If $b \in B (i)$, then either there is $j \in b$ such that $\eta_j \in \fct{\delta}{\lambda}$, or $b = \pred (i)$.
  \end{enumerate}

  There are three cases:

  \begin{itemize}
  \item $\{\eta_j : j < i\} = \{\nu_j : j < \beta\}$. Then we are done and let $\alpha := i$.
    \item If there is $b \in B (i)$ such that for some $j < \beta$, $\bigcup_{k \in b} \eta_k$ is an initial segment of $\nu_j$ but $\nu_j \notin \{\eta_k \mid k \in b\}$, then pick any such $b$ and the least such $j$, and let $\eta_i := \nu_j$.
    \item Otherwise, let $j < \beta$ be least such that $\nu_j \neq \eta_k$ for any $k < i$. Let $\eta_i := \nu_j$.
  \end{itemize}

  It is straightforward to see that this works.  
\end{proof}

We can now prove that $A_{i,j}$ is closed:

\begin{lem}\label{kb-lem}
  Let $\mathcal{T} := (\seq{M_i : i < \alpha}, N, \alpha, \tleq)$
  be a continuous enumerated tree of models. If:

  \begin{enumerate}
    \item $(\alpha, \tleq)$ is in pre-order.
    \item For any $b \in B (\alpha)$, $b$ is bounded.
  \end{enumerate}

  Then for any $i \tleq j < \alpha$, $A_{i, j} = \{k < j \mid \pred_{\tleq} (k) \cap \pred_{\tleq} (j) \subseteq \pred_{\tleq} (i)\}$ (see Definition \ref{indep-def}) is closed (see Definition \ref{closed-def}).
\end{lem}
\begin{proof}
  Let $b \in B (A_{i, j})$. We have to see that $\bigcup_{k \in b} M_k \lea N$. Now either $b = \pred_{\tleq} (i)$, in which case $\bigcup_{k \in b} M_k = M_i \lea N$, or $b \not \subseteq \pred_{\tleq} (j)$. In this case, it is easy to check that $b \in B (j)$ (otherwise we could just extend the branch), so since $(\alpha, \tleq)$ is in pre-order, either $b = \pred (j)$ or $b \in B (\alpha)$. The first case was dealt with before and in the second case, $b$ is bounded so has a maximum $j'$ (otherwise it would not be in $B(\alpha)$) and so $\bigcup_{k \in b} M_k = M_{j'} \lea N$.
\end{proof}

We can now prove that any reasonable tree can be ``made independent'' (and further, it will generate its ambient model $N$). This can be seen as a generalization of the existence axiom (see Definition \ref{axfr-def}.(\ref{axfr-def-existence})). Note that generalized symmetry is used in the proof.

\begin{lem}\label{tree-constr-lem}
  Assume that $\cl$ is algebraic and we are given a resolvable continuous enumerated tree of models $\mathcal{T}^0 := (\seq{M_i^0 : i < \alpha}$, $N^0$, $\alpha, \tleq)$. If:

  \begin{enumerate}
    \item $(\alpha, \tleq)$ is in pre-order.
    \item For any $b \in B (\alpha)$, $b$ is bounded.
  \end{enumerate}

  Then we can find $\seq{M_i : i < \alpha}$, $N$, and $\seq{f_i : i < \alpha}$ such that:

  \begin{enumerate}
    \item $\mathcal{T} := (\seq{M_i : i < \alpha}, N, \alpha, \tleq)$ is a resolvable \emph{independent} continuous enumerated tree.
    \item\label{tree-constr-2} For all $i, j < \alpha$, $f_i : M_i^0 \cong M_i$ and $i \tleq j$ implies $f_i \subseteq f_j$.
    \item $N = M^\alpha := \cl^N (\bigcup_{i < \alpha} M_i)$.
  \end{enumerate}
\end{lem}
\begin{proof}
  We build $\seq{N_i : i < \alpha}$, $\seq{M_i : i < \alpha}$, $\seq{f_i : i < \alpha}$ such that:

  \begin{enumerate}
    \item $\seq{N_i : i \le \alpha}$ is increasing.
    \item $\seq{f_i : i < \alpha}$ satisfies (\ref{tree-constr-2}).
    \item For all $i \in (0, \alpha)$, $\mathcal{T}_i := (\seq{M_j : j < i}, \bigcup_{j < i} N_j, i, \tleq)$ is a resolvable independent continuous enumerated tree of models.
    \item For all $i < \alpha$, $N_i = \cl^{N_i} (M_0 \cup \bigcup_{j < i} M_j)$ ($ = M^i$).
  \end{enumerate}

  This is enough, as we can then take $N := \bigcup_{i < \alpha} N_i$. This is possible. When $i = 0$, set $N_0 := M_0 := M_0^0$, $f_0 := \text{id}_{M_0^0}$. Now assume that $i > 0$. Let $N_i' := \bigcup_{j < i} N_j$. There are two cases:

  \begin{itemize}
    \item \underline{Case 1: $\pred (i)$ has a maximum}: Let $j := \max (\pred (i))$. Use the existence axiom (Definition \ref{axfr-def}.(\ref{axfr-def-existence})) to find $f_i$ extending $f_j$ and $N_i \gea N_i'$ so that $f_i : M_i^0 \cong M_i$, $\nfs{M_j}{M_i}{N_i'}{N_i}$, and $N_i = \cl^{N_i} (M_i \cup N_i')$. It is easy to check that this works.
    \item \underline{Case 2: $\pred (i)$ does not have a maximum}: Let $M_i' := \bigcup_{j \tlt i} M_j$, $(M_i^0)' := \bigcup_{j \tlt i} M_j^0$, $f_i' := \bigcup_{j \tlt i} f_j$. Let $M_i''$, $g: M_i^0 \cong M_i''$ be such that $g$ extends $f_i'$.

      Let $\delta := \otp (\pred (i))$. Note that $\delta$ is a limit ordinal. Let $\seq{i_j : j < \delta}$ list $\pred (i)$ in increasing order. For $j < i$, let $u_j := A_{i,j}$, where $A_{i,j}$ is as in Definition \ref{indep-def}. Note that $\bigcup_{j < \delta} u_j = i$. By Lemma \ref{kb-lem}, $u_j$ is closed in $\mathcal{T}^0$, hence (taking the image of $\mathcal{T}^0$ by $\bigcup_{j < i} f_j$) in $\mathcal{T}_i$. We use Fact \ref{basic-axfr-props}.(\ref{basic-props-5}) with $M_j$, $N_j$, $M$ there standing for $M_{i_j}$, $M^{u_{i_j}}$, $M_i''$ here. The hypotheses are satisfied by Theorem \ref{generalized-sym} (applied to $\mathcal{T}_i$) and monotonicity. We obtain $N_i \in \K$ and a map $f: M_i'' \xrightarrow[M_i']{} N_i$ such that for all $j < \delta$:

      \begin{enumerate}
      \item $N_{i_j} \lea N_i$.
      \item $\nfs{M_{i_j}}{f[M_i'']}{M^{u_{i_j}}}{N_i}$.
      \item $N_i = \cl^{N_i} (f[M_i''] \cup M^i)$.
      \end{enumerate}

      Let $f_i := f \circ g$ and let $M_i := f[M_i'']$. This works by the above properties.      
  \end{itemize}
\end{proof}

A specialization of Lemma \ref{tree-constr-lem} yields the main theorem of this section.

\begin{thm}[Tree construction]\label{tree-constr-thm}
  Assume that $\cl$ is algebraic. Let $\delta$ be a limit ordinal and let $\lambda \ge 2$ be a cardinal. Let $\seq{M_i : i \le \delta}$ be an increasing chain (we do \emph{not} need to assume that the models have size $\lambda$).

  If $\seq{M_i : i < \delta}$ is continuous but $\bigcup_{i < \delta} M_i \not \lea M_\delta$ (so $\delta$ is the least failure of smoothness for the chain $\seq{M_i : i < \delta}$), then there is $\seq{M_\eta \mid \eta \in \fct{\le \delta}{\lambda}}$, $\seq{f_\eta \mid \eta \in \fct{\le \delta}{\lambda}}$ and $N \in \K$ such that for all $\eta, \nu \in \fct{\le \delta}{\lambda}$:

  \begin{enumerate}
    \item\label{tree-constr-01} $M_\eta \lea N$, $f_\eta : M_{\ell (\eta)} \cong M_\eta$.
    \item\label{tree-constr-02} If $\eta$ is an initial segment of $\nu$, then $M_{\eta} \lea M_{\nu}$ and $f_\eta \subseteq f_\nu$.
    \item\label{tree-constr-03} If $\eta \neq \nu$ have length $\delta$ and $\alpha < \delta$ is least such that $\eta \rest (\alpha + 1) \neq \nu \rest (\alpha + 1)$, then $\nfs{M_{\nu \rest \alpha}}{M_{\eta}}{M_{\nu}}{N}$.
  \end{enumerate}
\end{thm}
\begin{proof}
  By Lemma \ref{kb-constr}, we can find an enumeration $\seq{\eta_i : i < \alpha}$ of $\fct{\le \delta}{\lambda}$ such that $(\alpha, \tleq)$ is a continuous enumerated tree in pre-order and $i \tleq j < \alpha$ implies that $\eta_i$ is an initial segment of $\eta_j$. For $i < \alpha$, let $M_i^0 := M_{\ell (\eta_i)}$ and let $N^0 := M_{\delta}$. Then it is straightforward to check that $\mathcal{T}^0 := (\seq{M_i^0 : i < \alpha}, N^0, \alpha, \tleq)$ satisfies the hypotheses of Lemma \ref{tree-constr-lem}. We obtain $\seq{M_i : i < \alpha}$, $N$, and $\seq{f_i : i < \alpha}$ there that correspond to $\seq{M_{\eta_i} : i < \alpha}$, $N$, and $\seq{f_{\eta_i} : i < \alpha}$ here. Since the resulting tree is independent, we obtain the independence condition via Lemma \ref{two-branches}.
\end{proof}

\section{Structure theory of universal classes}\label{uc-structure-sec}

In this section, we precisely state a result of Shelah saying that for a universal class $K$ which does not have the order property there is an ordering $\le$ so that $\K^0 := (K, \le)$ is a weak AEC satisfying $\Axfr$ (see Definition \ref{k-0-def}). To simplify matters, we partition $\K^0$ into disjoint classes, each of which has joint embedding, pick an appropriate such class and name it $\K^\ast$ (Definition \ref{k-ast-def}). We then use the tree construction theorem (Theorem \ref{tree-constr-thm}) to show that failure of smoothness in $\K^\ast$ implies unstability at certain cardinals (see Theorem \ref{smoothness-unstable}).

We start by specializing the order property from \cite[Definition V.A.1.1]{shelahaecbook2} to the quantifier-free version for universal classes:

\begin{defin}\label{op-def}
  A universal class $\K$ has the \emph{order property of length $\chi$} if there exists a quantifier-free first-order formula $\phi (\bx, \by, \bz)$, a model $M \in K$, a sequence $\bc \in \fct{\ell (\bz)}{|M|}$, and sequences $\seq{\ba_i : i < \chi}$, $\seq{\bb_i : i < \chi}$ from $M$ (with $\ell (\ba_i) = \ell (\bx)$, $\ell (\bb_i) = \ell (\by)$ for all $i < \chi$) so that for all $i, j < \chi$, $M \models \phi[\ba_i; \bb_j; \bc]$ if and only if $i < j$. We say that $\K$ has the \emph{order property} if it has the order property of length $\chi$ for all cardinals $\chi$.
\end{defin}
\begin{remark}
In the next section, we will show (Lemma \ref{no-op}) that categoricity in some $\lambda > \LS (\K)$ implies failure of the order property.
\end{remark}

The following result is proven (in a more general form) in \S2 of \cite{grsh222}.

\begin{fact}\label{op-length}
  Let $\K$ be a universal class. If $\K$ does not have the order property, then there exists $\chi < \hanf{\K}$ (recall Definition \ref{hanf-aec-def}) such that $\K$ does not have the order property of length $\chi$.
\end{fact}

From failure of the order property, Shelah shows that there exists a certain ordering $\le^{\chi^+, \mu^+}$ on $K$ such that $(K, \le^{\chi^+, \mu^+})$ satisfies $\Axfr$ (recall Definition \ref{axfr-def}). We now proceed to define this ordering.

\begin{defin}[Averages, V.A.2 in \cite{shelahaecbook2}]
  Let $\K$ be a universal class. Let $M \in \K$, let $I$ be an index set, and let $\BI := \seq{\ba_i : i \in I}$ be a sequence of elements of $M$ of the same finite arity $n < \omega$. Let $\chi \le \mu$ be infinite cardinals such that\footnote{We sometimes think of $\BI$ as just the set of its elements (i.e.\ as if it was only $\ran{\BI}$), e.g.\ we write $|\BI|$ instead of $|\ran{\BI}|$ and $\BI \subseteq \fct{n}{A}$ instead of $\ran{\BI} \subseteq \fct{n}{A}$.} $|\BI| \ge \chi$.

  \begin{enumerate}
  \item For $A \subseteq |M|$, we let $\Av_{\chi} (\BI / A; M)$ (the $\chi$-average of $\BI$ over $A$ in $M$) be the set of quantifier-free first-order formulas $\phi (\bx)$ over $A$ such that $\ell (\bx) = n$ and $|\{i \in I \mid M \models \neg \phi[\ba_i]\}| < \chi$.
  \item We say that $\BI$ is \emph{$(\chi, \mu)$-convergent in $M$} if $|\BI| \ge \mu$ and for every $A \subseteq |M|$, $p := \Av_{\chi} (\BI / A; M)$ is complete over $A$ (i.e.\ for every quantifier-free formula $\phi (\bx)$ over $A$ with $\ell (\bx) = n$, either $\phi (\bx) \in p$ or $\neg \phi (\bx) \in p$).
  \item Let $A, B \subseteq |M|$ and let $p$ be a set of quantifier-free formulas over $B$ (all of the same arity $n < \omega$). We say that $p$ is \emph{$(\chi, \mu)$-averageable over $A$ in $M$} if there exists a sequence $\BI \subseteq \fct{n}{A}$ that is $(\chi, \mu)$-convergent in $M$ and with $p = \Av_{\chi} (\BI / B; M)$.
\end{enumerate}
\end{defin}
\begin{remark}
  In the above notation, the usual notion of average from the first-order framework \cite[Definition III.1.5]{shelahfobook} can be written $\Av_{\aleph_0} (\BI / A; \sea)$, modulo the fact that here all the formulas are quantifier-free.
\end{remark}
\begin{remark}[Monotonicity]\label{monot-average-rmk} \
  \begin{enumerate}
  \item\label{monot-rmk-1} Since the formulas under consideration are quantifier-free, we have the following monotonicity properties: if $M_0 \subseteq M$ and $A, \BI \subseteq |M_0|$, then $\Av_{\chi} (\BI / A; M_0) = \Av_{\chi} (\BI / A; M)$. Similarly, if $\BI$ is $(\chi, \mu)$-convergent in $M$, then it is $(\chi, \mu)$-convergent in $M_0$, and if $A \subseteq B \subseteq |M_0|$ and $p$ is a quantifier-free type over $B$ that is $(\chi, \mu)$-averageable over $A$ in $M$, then it is $(\chi, \mu)$-averageable over $A$ in $M_0$.
  \item If $p$ over $B$ as in (\ref{monot-rmk-1}) is $(\chi, \mu)$-averageable over $A$ in $M$, then whenever $A \subseteq A' \subseteq B_0 \subseteq B$, we have that $p \rest B_0$ is $(\chi, \mu)$-averageable over $A'$ in $M$.
  \end{enumerate}
\end{remark}

\begin{defin}[V.A.4.1 in \cite{shelahaecbook2}]\label{average-order}
  Let $\K$ be a universal class and let $\chi \le \mu$ be infinite cardinals. For $M, N \in \K$, we write $M \le^{\chi, \mu} N$ if $M \subseteq N$ and for every $\bc \in \fct{<\omega}{|N|}$, the quantifier-free type of $\bc$ over $M$ in $N$, $\tp_{\text{qf}} (\bc / M; N)$, is $(\chi, \mu)$-averageable over $M$ .
\end{defin}

Note that if $M, N \in \K_{<\mu}$, then we never have $M \le^{\chi, \mu} N$. From now on we assume:

\begin{hypothesis}\label{structure-sec-hyp} \
  \begin{enumerate}
    \item $\K = (K, \subseteq)$ is a universal class with arbitrarily large models.
    \item $\chi \ge \LS (\K)$ is such that $\K$ does not have the order property of length $\chi^+$.
    \item Set $\mu := 2^{2^{\chi}}$.
  \end{enumerate}
\end{hypothesis}

\begin{defin}\label{k-0-def}
  Let $\K^0 := (K, \le^{\chi^+, \mu^+})$.
\end{defin}

The following is the key structure theorem for universal classes: from failure of the order property, Shelah \cite[Chapter V.B]{shelahaecbook2} shows that one can make $\K^0$ into a weak AEC satisfying $\Axfr$. Note that by Fact \ref{op-length} one can take $\chi, \mu < \ehanf{\LS (\K)}$.

\begin{fact}\label{axfr-fact} \
  \begin{enumerate}
    \item $\K^0$ is a weak AEC with $\LS (\K^0) \le \mu^+$.
    \item For $M \in K$ and $A \subseteq |M|$, let $\cl^M (A)$ be the closure of $A$ under the functions of $M$. We can define a $4$-ary relation $\nf$ on $K$ by $\nfs{M_0}{M_1}{M_2}{M_3}$ if and only if all of the following conditions are satisfied:
      \begin{enumerate}
      \item $M_0 \leap{\K^0} M_1$ and $M_0 \leap{\K^0} M_2$.
      \item $M_1 \subseteq M_3$ and $M_2 \subseteq M_3$.
      \item $\cl^{M_3} (M_1 \cup M_2) \leap{\K^0} M_3$.
      \item For any $\bc \in \fct{<\omega}{|M_1|}$, $\tp_{\text{qf}} (\bc / M_2; M_3)$ is $(\chi^+, \mu^+)$-averageable over $M_0$.
      \end{enumerate}

      We then have that $(\K^0, \nf, \cl)$ satisfies $\Axfr$. Moreover $\cl$ is algebraic (see Definition \ref{cl-algebraic}) and $\nf$ is $\mu^+$-based (see Definition \ref{based-def}).
  \end{enumerate}
\end{fact}
\begin{proof}
  That $(\K^0, \nf, \cl)$ satisfies $\Axfr$ and has Löwenheim-Skolem-Tarski number bounded by $\mu^+$ is the content of \cite[V.B.2.9]{shelahaecbook2}. Since $\cl$ is just closure under the functions, it is clearly algebraic. That $\nf$ is $\mu^+$-based is observed (but not explicitly proven) in \cite[V.C.5.7]{shelahaecbook2}. We give the proof here.

  \underline{Claim}: $\nf$ is $\mu^+$-based.
  
  \underline{Proof of Claim}:
  
  First, we show:

  \underline{Subclaim}: For any cardinal $\lambda$, $\K^0$ is $(\le \lambda, \mu^+)$-smooth. That is, if $\seq{M_i : i < \mu^+}$ is increasing in $\K^0$ and $M \in \K^0$ is such that $M_i \leap{\K^0} M$ for all $i < \mu^+$, then $\bigcup_{i < \mu^+} M_i \leap{\K^0} M$.

  \underline{Proof of Subclaim}: 

  In \cite[V.A.4.4]{shelahaecbook2}, it is shown that for any $N, N' \in \K^0$, $N \leap{\K^0} N'$ if and only if $N \lee_{\Delta} N'$, where $\Delta$ is a certain fragment of $\Ll_{\mu^+, \mu^+}$. The result now follows from the basic properties of $\Delta$-elementary substructure. $\dagger_{\text{Subclaim}}$

  Let $M \leap{\K^0} M^\ast$ and let $A \subseteq |M^\ast|$ be given. By definition of $\leap{\K^0} = \le^{\chi^+, \mu^+}$, for each $\bc \in \fct{<\omega}{|M^\ast|}$ there exists $\BI^{\bc} \subseteq \fct{\ell (\bc)}{|M|}$ that is $(\chi^+, \mu^+)$-convergent and so that $\Av (\BI^{\bc} / M; M^\ast) = \tp_{\text{qf}} (\bc / M; M^\ast)$. Without loss of generality, $|\BI^{\bc}| \le \mu^+$. 

  We build increasing $\seq{M_i^0 : i < \mu^+}$, $\seq{M_i^1 : i < \mu^+}$ such that for all $i < \mu^+$:

  \begin{enumerate}
    \item\label{based-claim-1} $M_i^0 \leap{\K^0} M$.
    \item $M_i^0 \leap{\K^0} M_i^1 \leap{\K^0} M^\ast$.
    \item $\|M_i^1\| \le |A| + \mu^+$.
    \item\label{based-claim-4} $|M| \cap |M_i^1| \subseteq |M_0^{i + 1}|$.
    \item\label{based-claim-5} For all $\bc \in \fct{<\omega}{M_i^1}$, $\BI^{\bc} \subseteq |M_{i + 1}^{0}|$.
  \end{enumerate}

  This is enough: let $N_0 := \bigcup_{i < \mu^+} M_i^0$, $N_1 := \bigcup_{i < \mu^+} M_i^1$. By the claim, $N_0 \leap{\K^0} N_1 \leap{\K^0} M^\ast$ and by requirements (\ref{based-claim-1}) and (\ref{based-claim-4}), $M \cap N_1 = N_0$. Finally, $\nfs{N_0}{N_1}{M}{M^\ast}$ by definition of $\BI^{\bc}$ and requirement (\ref{based-claim-5}).

  This is possible: assume that $\seq{M_j^\ell : j < i}$ have been defined for $\ell = 0,1$. Let $M_{i, 0}^0 := \bigcup_{j < i} M_j^0$, $M_{i, 0}^1 := \bigcup_{j < i} M_j^1$. Use that $\LS (\K^0) \le \mu^+$ to pick $M_i^0$ such that $M_i^0 \leap{\K^0} M$, $|M| \cap M_{i, 0}^1 \subseteq |M_i^0|$, $\BI^{\bc} \subseteq |M_i^0|$ for all $\bc \in \fct{<\omega}{|M_{i, 0}^1|}$, and $\|M_i^0\| \le |A| + \mu^+$. Note that by coherence, $M_j^0 \leap{\K^0} M_i^0$. Now pick $M_i^1$ such that $M_i^1 \leap{\K^0} M^\ast$, $A \cup |M_{i, 0}^1| \cup |M_i^0| \subseteq M_i^1$, and $\|M_i^1\| \le |A| + \mu^+$. It is easy to check that this satisfies all the requirements. $\dagger_{\text{Claim}}$
\end{proof}

Note that $\K^0$ has amalgamation (Remark \ref{ap-rmk}). However, we do not know if it satisfies joint embedding, so we partition $\K^0$ into disjoint AECs, each of which has joint embedding. We will then concentrate on just one of these AECs. This trick appears in \cite[Section II.3]{sh300-orig}.

\begin{defin}\label{k-ast-def}
  For $M, N \in \K^0$, write $M \sim N$ if they can be $\leap{\K^0}$-embedded inside a common model. This is an equivalence relation and the equivalence classes partition $\K^0$ into disjoint weak AECs $\seq{\K_i^0 : i \in I}$ that have amalgamation and joint embedding. There is only a set of such classes, so there exists\footnote{There could be many and for our purpose the choice of $i$ does not matter. Moreover $i$ is unique if $\K$ is categorical in some $\lambda \ge \mu^+$.} $i \in I$ such that $\K_i^0$ has arbitrarily large models. Let $\K^\ast := \left(\K_i^0\right)_{\ge \mu^+}$.
\end{defin}

From now on, we will work with $\K^\ast$. We note a few trivial properties of independence there:

\begin{lem}\label{k-ast-prop} \
  \begin{enumerate}
    \item\label{k-ast-prop-1} $\K^\ast$ is a weak AEC with amalgamation, joint embedding, and arbitrarily large models.
    \item\label{k-ast-prop-2} $\LS (\K^\ast) = \mu^+$.
    \item\label{k-ast-prop-3} $\K$ and $\K^\ast$ are compatible (recall Definition \ref{compatible-def}).
    \item\label{k-ast-prop-4} $(\K^\ast, \nf \rest \K^\ast, \cl)$ satisfies $\Axfr$, where for $M \in \K^\ast$, $\cl^M$ is closure under the functions of $M$ and $\nf \rest \K^\ast$ is the natural restriction of $\nf$ (from Fact \ref{axfr-fact}) to $\K^\ast$.
  \end{enumerate}
\end{lem}
\begin{proof}
  Straightforward.
\end{proof}
\begin{notation}
We abuse notation and write $\nf$ for $\nf \rest \K^\ast$ (where again $\nf$ is from Fact \ref{axfr-fact}).
\end{notation}

\begin{lem}\label{basic-indep-props} \
  \begin{enumerate}
    \item\label{basic-indep-coher} If $\nfcl{M_0}{A}{B}{M_3}$ and $\bc \in \fct{<\omega}{A}$, then $\tp_{\text{qf}} (\bc / B; M_3)$ is $(\chi^+, \mu^+)$-averageable over $M_0$.
    \item\label{basic-indep-cl} $\cl$ is algebraic.
    \item\label{basic-indep-based} $\nf$ is $\LS (\K)$-based.
  \end{enumerate}
\end{lem}
\begin{proof}
  (\ref{basic-indep-coher}) follows directly from the definition of $\nf$. For the rest, $\cl$ is algebraic because $\cl$ satisfies this property in $\K^0$ (Fact \ref{axfr-fact}). Similarly in $\K^0$, $\nf$ is $\mu^+$-based (Fact \ref{axfr-fact}) and it is straightforward to check that this carries over to $\K^\ast$.
\end{proof}

Next, we study what happens if smoothness fails in $\K^\ast$. Recall that our goal is to see that this is incompatible with categoricity (in a high-enough cardinal). Shelah has shown \cite[V.C.2.6]{shelahaecbook2}, that failure of smoothness implies that $\K^\ast$ has $2^\lambda$-many nonisomorphic models at every high-enough \emph{regular} cardinal $\lambda$. So in particular $\K^\ast$ cannot be categorical in a regular cardinal. However we are also interested in the singular case. Shelah states as an exercise \cite[V.C.4.13]{shelahaecbook2} that $\K^\ast$ has (at least) $2^{<\lambda}$-many nonisomorphic models if $\lambda$ is singular. However we have been unable to prove it.

Instead, we aim to see that failure of smoothness implies that $\K^\ast$ has many \emph{types}, i.e.\ it is Galois unstable in some suitable cardinals. This will contradict Lemma \ref{stable-prop}. The argument is similar to \cite[V.E.3.15]{shelahaecbook2}, which shows that failure of superstability (in the sense that there is an increasing chain $\seq{M_i : i < \delta}$ and a type $p \in \gS (\bigcup_{i < \delta} M_i)$ that forks over every $M_i$, $i < \delta$) implies unstability at suitable cardinals. The extra difficulty here is that smoothness fails, but the hard work in constructing the tree has already been done in Theorem \ref{tree-constr-thm}.

First observe that any failure of smoothness must be witnessed by a small chain:

\begin{lem}\label{smoothness-upward}
  If $\K^\ast$ is $(\le \LS (\K^\ast), \le \LS (\K^\ast)^+)$-smooth (recall Definition \ref{smooth-def}), then $\K^\ast$ is smooth, i.e.\ it is an AEC.
\end{lem}
\begin{proof}
  By Lemma \ref{basic-indep-props}.(\ref{basic-indep-based}), $\K^\ast$ is $\LS (\K^\ast)$-based, so apply Fact \ref{smoothness-upward-fact}.
\end{proof}

We now show that failure of smoothness implies unstability at some not too high cardinal. A technical subtlety is that we can only show $(<\omega)$-unstability, i.e.\ there are many types of some fixed finite length. In this framework, we do not know whether this implies that there are also many types of length one (see also Remark \ref{omega-stability-rmk}).

\begin{thm}\label{smoothness-unstable}
  Assume that $\K^\ast$ is not $(\le \LS (\K^\ast), \le \LS (\K^\ast)^+)$-smooth. Let $\kappa \le \LS (\K^\ast)^+$ be least such that $(\le \LS (\K^\ast), \le \kappa)$-smoothness fails. If $\lambda \ge \LS (\K^\ast)^+$ is such that $\lambda = \lambda^{<\kappa}$ and $\lambda < \lambda^{\kappa}$, then $\K^\ast$ is $(<\omega)$-unstable in $\lambda$.
\end{thm}
\begin{proof}
  Fix an increasing sequence $\seq{M_i : i \le \kappa}$ such that $\|M_i\| \le \LS (\K^\ast)^+$ for all $i \le \kappa$ and $\bigcup_{i < \kappa} M_i \not \leap{\K^\ast} M_{\kappa}$. Without loss of generality (using minimality of $\kappa$) the sequence is continuous below $\kappa$, i.e.\ $M_i = \bigcup_{j < i} M_j$ for every $i < \kappa$. Let $N \in \K^\ast$ and $\seq{M_{\eta}, f_{\eta} \mid \eta \in \fct{\le \kappa}{\lambda}}$ be as given by Theorem \ref{tree-constr-thm} (where $\delta, \K$ there stands for $\kappa, \K^\ast$ here; note that $\cl$ is algebraic by Lemma \ref{basic-indep-props}.(\ref{basic-indep-cl}) so the hypotheses of the theorem hold). 

  By definition of $\leap{\K^\ast}$ (so really of $\leap{\K^0}$, see Definitions \ref{k-0-def} and \ref{average-order}), we have that $\bigcup_{i < \kappa} M_i \not \le^{\chi^+, \mu^+} M_\kappa$. By definition of $\le^{\chi^+, \mu^+}$, there exists $\bc \in \fct{<\omega}{|M_\kappa|}$ such that $q := \tp_{\text{qf}} (\bc / \bigcup_{i < \kappa} M_i; M_\kappa)$ is not $(\chi^+, \mu^+)$-averageable over $\bigcup_{i < \kappa} M_i$ in $M_{\kappa}$. For $\eta \in \fct{\kappa}{\lambda}$, let $\bc_{\eta} := f_{\eta} (\bc)$. 

  Note that by (\ref{tree-constr-01}) in Theorem \ref{tree-constr-thm}, for all $\eta \in \fct{\le \kappa}{\lambda}$, $\|M_{\eta}\| = \|M_{\ell (\eta)}\| \le \LS (\K^\ast)^+ \le \lambda$, so fix $M \leap{\K^\ast} N$ such that $\|M\| = \lambda$ and $\bigcup_{\eta \in \fct{<\kappa}{\lambda}} |M_{\eta}| \subseteq |M|$. For $\eta \in \fct{\kappa}{\lambda}$, let $p_{\eta} := \gtp_{\K^\ast} (\bc_{\eta} / M; N)$.

  Because $\lambda < \lambda^{\kappa}$, it is enough to prove the following: 

  \underline{Claim}: For $\eta, \nu \in \fct{\kappa}{\lambda}$, if $\eta \neq \nu$, then $p_{\eta} \neq p_{\nu}$.

  \underline{Proof of claim}: Let $\alpha < \kappa$ be least such that $\eta \rest (\alpha + 1) \neq \nu \rest (\alpha + 1)$. By (\ref{tree-constr-03}) in Theorem \ref{tree-constr-thm} and the monotonicity property of $\nfm$ (see Lemma \ref{nfcl-very-basic}) we have that $\nfcl{M_{\nu \rest \alpha}}{\bc_{\eta}}{M_{\nu}}{N}$. By monotonicity again, $\nfcl{M_{\nu \rest \alpha}}{\bc_{\eta}}{\bigcup_{\beta < \kappa} M_{\nu \rest \beta}}{N}$. Now assume for a contradiction that $p_{\eta} = p_{\nu}$. Then by monotonicity and invariance, $\nfcl{M_{\nu \rest \alpha}}{\bc_{\nu}}{\bigcup_{\beta < \kappa} M_{\nu \rest \beta}}{N}$ so $\nfcl{M_{\nu \rest \alpha}}{\bc_{\nu}}{\bigcup_{\beta < \kappa} M_{\nu \rest \beta}}{M_{\nu}}$. Applying $f_{\nu}^{-1}$ to this, we get that $\nfcl{M_\alpha}{\bc}{\bigcup_{i < \kappa} M_i}{M_{\kappa}}$. In particular, by Lemma (\ref{basic-indep-props}).(\ref{basic-indep-coher}), $q$ is $(\chi^+, \mu^+)$-averageable over $M_\alpha$ in $M_{\kappa}$. By Remark \ref{monot-average-rmk}, $q$ is $(\chi^+, \mu^+)$-averageable over $\bigcup_{i < \kappa} M_i$ in $M_{\kappa}$. This contradicts the choice of $\bc$. $\dagger_{\text{Claim}}$.
\end{proof}

\section{Categoricity in universal classes}

In this section, we derive the main theorem of this paper. First, we explain why, in a universal class, categoricity (in some $\lambda > \LS (\K)$) implies failure of the order property. Note that Shelah argues \cite[Claim V.B.2.6]{shelahaecbook2} that if $K$ has the order property, then it has $2^{\mu}$-many models of size $\mu$ (for any $\mu > \LS (\K)$). In particular, this violates categoricity but Shelah's construction of many models is very technical and when categoricity is assumed there is an easier proof. Note that we do not even need to work with Galois types and can use syntactic (first-order) quantifier-free types instead.

\begin{lem}\label{no-op}
  Assume that a universal class $\K$ is categorical in a $\lambda > \LS (\K)$. Then $\K$ does not have the order property (recall Definition \ref{op-def}).
\end{lem}
\begin{proof}
  If $\K$ does not have arbitrarily large models, then $\K$ does not have the order property. Now assume that $\K$ has arbitrarily large models. We can use Ehrenfeucht-Mostowski models and the standard argument (due to Morley, see \cite[Theorem 3.7]{morley-cip}) shows that if $M \in K_\lambda$, $\mu \in [\LS (\K), \lambda)$, and $A \subseteq |M|$ is such that $|A| \le \mu$, then $M$ realizes at most $\mu$-many first-order syntactic quantifier-free types over $A$. However if $\K$ had the order property, we would be able to build a set $A \subseteq |M|$ with $|A| \le \LS (\K)$ but with at least $\LS (\K)^+$ (syntactic quantifier-free) types over $A$ realized in $M$ (using Dedekind cuts, see e.g.\ the proof of \cite[Fact 5.13]{bgkv-apal}). This is a contradiction.
\end{proof}

Next, we deduce more structure from categoricity:

\begin{thm}\label{universal-class-structure}
  Let $\K$ be a universal class. If $\K$ is categorical in some $\lambda \ge \beth_{\hanf{\K}}$, then there exists $\K^\ast$ such that:

  \begin{enumerate}
    \item\label{structure-1} $\K^\ast$ is an AEC.
    \item\label{structure-2} $\LS (\K) \le \LS (\K^\ast) < \hanf{\K}$.
    \item\label{structure-3} $\K$ and $\K^\ast$ are compatible (recall Definition \ref{compatible-def}).
    \item\label{structure-4} $\K^\ast$ has amalgamation, joint embedding, and arbitrarily large models.
    \item\label{structure-5} $\K^\ast$ is $\LS (\K^\ast)$-tame.
  \end{enumerate}
\end{thm}
\begin{proof}
  Let $\K$ be a universal class and let $\lambda \ge \beth_{\hanf{\K}}$ be such that $\K$ is categorical in $\lambda$. By Fact \ref{hanf-arb-large}, $\K$ has arbitrarily large models. By Lemma \ref{no-op}, $\K$ does not have the order property. By Fact \ref{op-length}, we can fix $\chi \in [\LS (\K), \hanf{\K})$ such that $\K$ does not have the order property of length $\chi^+$. Thus Hypothesis \ref{structure-sec-hyp} is satisfied, and so Shelah's structure theorem for universal classes (Fact \ref{axfr-fact}) applies. Let $\K^\ast$ be as in Definition \ref{k-ast-def}. We have to check that it has all the required properties. First, $\K^\ast$ is a weak AEC with amalgamation, joint embedding, and arbitrarily large models (Lemma \ref{k-ast-prop}.(\ref{k-ast-prop-1})). Moreover (Lemma \ref{k-ast-prop}.(\ref{k-ast-prop-2})), $\LS (\K) \le \LS (\K^\ast) = \mu^+ = \left(2^{2^{\chi}}\right)^+ < \hanf{\K}$. Also, $\K$ and $\K^\ast$ are compatible (Lemma \ref{k-ast-prop}.(\ref{k-ast-prop-3})). This takes care of (\ref{structure-2}), (\ref{structure-3}), and (\ref{structure-4}) in the statement of Theorem \ref{universal-class-structure}. Combining Lemma \ref{basic-indep-if-based}.(\ref{basic-indep-tame}) and Lemma \ref{basic-indep-props}, we obtain that $\K^\ast$ is $\LS (\K^\ast)$-tame, so (\ref{structure-5}) also holds. 

    It remains to see (\ref{structure-1}): $\K^\ast$ is an AEC, i.e.\ it satisfies the smoothness axiom. Suppose not. Then by Lemma \ref{smoothness-upward}, there is a small counter-example: $\K^\ast$ is not $(\le \LS (\K^\ast), \le \LS (\K^\ast)^+)$-smooth. Let $\kappa \le \LS (\K^\ast)^+$ be least such that $\K^\ast$ is not $(\le \LS (\K^\ast), \le \kappa)$-smooth. Note that $\kappa$ is regular. Let $\lambda_0 := \beth_{\kappa} (\LS (\K^\ast))$. Note:

        \begin{itemize}
        \item $\lambda_0 \ge \LS (\K^\ast)^+$.
        \item $\lambda_0 = \lambda_0^{<\kappa}$ and $\lambda_0 < \lambda_0^{\kappa}$ (because $\cf{\lambda_0} = \kappa$).
        \item Since $\kappa \le \LS (\K^\ast) < \hanf{\K}$, we have that $\lambda_0 \le \beth_{\LS (\K^\ast) + \kappa} < \beth_{\hanf{\K}} \le \lambda$. Similarly, $\lambda_0^+ < \lambda$.
        \end{itemize}

        By Lemma \ref{stable-prop} (where $\K^1, \K^2, \mu, \lambda$ there stand for $\K, \K^\ast, \lambda_0, \lambda$ here, note that we are using that $\lambda_0^+ < \lambda$), $\K^\ast$ is $(<\omega)$-stable in $\lambda_0$. However Theorem \ref{smoothness-unstable} (where $\lambda$ there stands for $\lambda_0$ here) says that $\K^\ast$ is $(<\omega)$-unstable in $\lambda_0$, a contradiction.
\end{proof}

Finally, we have all the results we need to prove the main theorem:

\begin{thm}\label{main-thm}
  Let $\K$ be a universal class. If $\K$ is categorical in \emph{some} $\lambda \ge \beth_{\hanf{\K}}$, then there exists $\chi < \beth_{\hanf{\K}}$ such that $\K$ is categorical in \emph{all} $\lambda' \ge \chi$. Moreover, $\K_{\ge \chi}$ has amalgamation.
\end{thm}
\begin{proof}
  Let $\K^\ast$ be as given by Theorem \ref{universal-class-structure}. In particular, $\K^\ast$ is tame and has amalgamation. By Fact \ref{univ-prime}, $\K$ has primes, so we can use Theorem \ref{ap-categ-equal}, compatibility, and the categoricity transfer theorem for tame AECs with primes (Fact \ref{categ-facts}.(\ref{categ-facts-2b})). That is, by Theorem \ref{abstract-compatible-categ} (where $\K^1$, $\K^2$ there stand for $\K$, $\K^\ast$ here), $\K^\ast$ is categorical in all $\lambda' \ge \chi := \hanf{\LS (\K^\ast)}$. By compatibility (recalling that $\LS (\K) \le \LS (\K^\ast)$), $\K$ is also categorical in all $\lambda' \ge \chi$. Finally, since $\LS (\K^\ast) < \hanf{\K}$, we have that $\chi = \hanf{\LS (\K^\ast)} = \beth_{\left(2^{\LS (\K^\ast)}\right)^+} < \beth_{\hanf{\K}}$. 

  For the moreover part, note that $\chi^{\LS (\K, \K^\ast)} = \chi^{\LS (\K^\ast)} = \chi$ so by Lemma \ref{equal-aecs-prop}, $\K_{\ge \chi} = \K^\ast_{\ge \chi}$. Since the latter has amalgamation, so does the former. 
\end{proof}
\begin{remark}\label{main-thm-rmk}
  In fact, $\K_{\ge \chi}$ satisfies much more than amalgamation. This is because $\K_{\ge \chi}$ is a locally universal class (see \cite[Definition 2.20]{ap-universal-v9}). Thus it is fully $\chi$-tame and short (see \cite[Corollary 3.8]{ap-universal-v9}) and admits a global notion of independence (for types over arbitrary sets) that is similar to forking in a first-order superstable theory (see \cite[Appendix C]{ap-universal-v9}).
\end{remark}
\begin{proof}[Proof of Theorem \ref{abstract-thm} and Corollary \ref{abstract-thm-2}]
  Let $\psi$ be a universal $\Ll_{\omega_1, \omega}$ sentence. The class $\K$ of models of $\psi$ is a universal class (Fact \ref{univ-charact}) with $\hanf{\K} = \beth_{\omega_1}$ (see Remark \ref{univ-delta} and Fact \ref{delta-facts}). Now apply Theorem \ref{main-thm}.
\end{proof}
\begin{remark}
  By Fact \ref{univ-charact} and Remark \ref{univ-delta}, Theorem \ref{abstract-thm} and Corollary \ref{abstract-thm-2} apply more generally to any universal class in a countable vocabulary.
\end{remark}

\bibliographystyle{amsalpha}
\bibliography{universal-classes-2}

\providecommand{\bysame}{\leavevmode\hbox to3em{\hrulefill}\thinspace}
\providecommand{\MR}{\relax\ifhmode\unskip\space\fi MR }
\providecommand{\MRhref}[2]{%
  \href{http://www.ams.org/mathscinet-getitem?mr=#1}{#2}
}
\providecommand{\href}[2]{#2}
\begin{thebibliography}{BGKV16}

\bibitem[Bal09]{baldwinbook09}
John~T. Baldwin, \emph{Categoricity}, University Lecture Series, vol.~50,
  American Mathematical Society, 2009.

\bibitem[BGKV16]{bgkv-apal}
Will Boney, Rami Grossberg, Alexei Kolesnikov, and Sebastien Vasey,
  \emph{Canonical forking in {A}{E}{C}s}, Annals of Pure and Applied Logic
  \textbf{167} (2016), no.~7, 590--613.

\bibitem[Bon14]{tamelc-jsl}
Will Boney, \emph{Tameness from large cardinal axioms}, The Journal of Symbolic
  Logic \textbf{79} (2014), no.~4, 1092--1119.

\bibitem[BV]{bv-survey-v4-toappear}
Will Boney and Sebastien Vasey, \emph{A survey on tame abstract elementary
  classes}, To appear in Beyond first order model theory. URL:
  \url{http://arxiv.org/abs/1512.00060v4}.

\bibitem[Cha68]{chang-pres}
C.C. Chang, \emph{Some remarks on the model theory of infinitary languages},
  The syntax and semantics of infinitary languages (Jon Barwise, ed.), Lecture
  Notes in Mathematics, vol.~72, Springer, 1968, pp.~36--63.

\bibitem[GK]{superior-aec}
Rami Grossberg and Alexei Kolesnikov, \emph{Superior abstract elementary
  classes are tame}, Preprint. URL:
  \url{http://www.math.cmu.edu/~rami/AtameP.pdf}.

\bibitem[Gro02]{grossberg2002}
Rami Grossberg, \emph{Classification theory for abstract elementary classes},
  Contemporary Mathematics \textbf{302} (2002), 165--204.

\bibitem[GS86]{grsh222}
Rami Grossberg and Saharon Shelah, \emph{On the number of nonisomorphic models
  of an infinitary theory which has the infinitary order property. {P}art {A}},
  The Journal of Symbolic Logic \textbf{51} (1986), no.~2, 302--322.

\bibitem[GV06a]{tamenessthree}
Rami Grossberg and Monica VanDieren, \emph{Categoricity from one successor
  cardinal in tame abstract elementary classes}, Journal of Mathematical Logic
  \textbf{6} (2006), no.~2, 181--201.

\bibitem[GV06b]{tamenessone}
\bysame, \emph{Galois-stability for tame abstract elementary classes}, Journal
  of Mathematical Logic \textbf{6} (2006), no.~1, 25--49.

\bibitem[GV06c]{tamenesstwo}
\bysame, \emph{Shelah's categoricity conjecture from a successor for tame
  abstract elementary classes}, The Journal of Symbolic Logic \textbf{71}
  (2006), no.~2, 553--568.

\bibitem[KLH16]{coloring-classes-jsl}
Alexei Kolesnikov and Christopher Lambie-Hanson, \emph{The {H}anf number for
  amalgamation of coloring classes}, The Journal of Symbolic Logic \textbf{81}
  (2016), no.~2, 570--583.

\bibitem[Kue08]{kueker2008}
David~W. Kueker, \emph{Abstract elementary classes and infinitary logics},
  Annals of Pure and Applied Logic \textbf{156} (2008), 274--286.

\bibitem[Mal69]{malitz-universal}
Jerome Malitz, \emph{Universal classes in infinitary languages}, Duke
  Mathematical Journal \textbf{36} (1969), no.~3, 621--630.

\bibitem[Mor65]{morley-cip}
Michael Morley, \emph{Categoricity in power}, Transactions of the American
  Mathematical Society \textbf{114} (1965), 514--538.

\bibitem[MS90]{makkaishelah}
Michael Makkai and Saharon Shelah, \emph{Categoricity of theories in
  ${L}_{\kappa,\omega}$, with $\kappa$ a compact cardinal}, Annals of Pure and
  Applied Logic \textbf{47} (1990), 41--97.

\bibitem[She]{sh394-updated}
Saharon Shelah, \emph{Categoricity for abstract classes with amalgamation
  (updated)}, Oct. 29, 2004 version. URL:
  \url{http://shelah.logic.at/files/394.pdf}.

\bibitem[She83a]{sh87a}
\bysame, \emph{Classification theory for non-elementary classes {I}: The number
  of uncountable models of $\psi \in {L}_{\omega_1, \omega}$. {P}art {A}},
  Israel Journal of Mathematics \textbf{46} (1983), no.~3, 214--240.

\bibitem[She83b]{sh87b}
\bysame, \emph{Classification theory for non-elementary classes {I}: The number
  of uncountable models of $\psi \in {L}_{\omega_1, \omega}$. {P}art {B}},
  Israel Journal of Mathematics \textbf{46} (1983), no.~4, 241--273.

\bibitem[She87a]{sh88}
\bysame, \emph{Classification of non elementary classes {II}. {A}bstract
  elementary classes}, Classification Theory (Chicago, IL, 1985) (John~T.
  Baldwin, ed.), Lecture Notes in Mathematics, vol. 1292, Springer-Verlag,
  1987, pp.~419--497.

\bibitem[She87b]{sh300-orig}
\bysame, \emph{Universal classes}, Classification theory (Chicago, IL, 1985)
  (John~T. Baldwin, ed.), Lecture Notes in Mathematics, vol. 1292,
  Springer-Verlag, 1987, pp.~264--418.

\bibitem[She90]{shelahfobook}
\bysame, \emph{Classification theory and the number of non-isomorphic models},
  2nd ed., Studies in logic and the foundations of mathematics, vol.~92,
  North-Holland, 1990.

\bibitem[She99]{sh394}
\bysame, \emph{Categoricity for abstract classes with amalgamation}, Annals of
  Pure and Applied Logic \textbf{98} (1999), no.~1, 261--294.

\bibitem[She09a]{shelahaecbook}
\bysame, \emph{Classification theory for abstract elementary classes}, Studies
  in Logic: Mathematical logic and foundations, vol.~18, College Publications,
  2009.

\bibitem[She09b]{shelahaecbook2}
\bysame, \emph{Classification theory for abstract elementary classes 2},
  Studies in Logic: Mathematical logic and foundations, vol.~20, College
  Publications, 2009.

\bibitem[SV99]{shvi635}
Saharon Shelah and Andr{\'e}s Villaveces, \emph{Toward categoricity for classes
  with no maximal models}, Annals of Pure and Applied Logic \textbf{97} (1999),
  1--25.

\bibitem[Tar54]{tarski-th-models-i}
Alfred Tarski, \emph{Contributions to the theory of models i}, Indagationes
  Mathematicae \textbf{16} (1954), 572--581.

\bibitem[Vasa]{downward-categ-tame-v6-toappear}
Sebastien Vasey, \emph{Downward categoricity from a successor inside a good
  frame}, Annals of Pure and Applied Logic, To appear. URL:
  \url{http://arxiv.org/abs/1510.03780v6}.

\bibitem[Vasb]{categ-primes-v4}
\bysame, \emph{Shelah's eventual categoricity conjecture in tame {A}{E}{C}s
  with primes}, Preprint. URL: \url{http://arxiv.org/abs/1509.04102v4}.

\bibitem[Vasc]{ap-universal-v9}
\bysame, \emph{Shelah's eventual categoricity conjecture in universal classes.
  {P}art {I}}, Preprint. URL: \url{http://arxiv.org/abs/1506.07024v9}.

\bibitem[Vas16a]{indep-aec-apal}
\bysame, \emph{Building independence relations in abstract elementary classes},
  Annals of Pure and Applied Logic \textbf{167} (2016), no.~11, 1029--1092.

\bibitem[Vas16b]{sv-infinitary-stability-afml}
\bysame, \emph{Infinitary stability theory}, Archive for Mathematical Logic
  \textbf{55} (2016), 567--592.

\end{thebibliography}

\end{document}